\documentclass[a4paper,10pt,reqno]{amsart}
\usepackage[latin1]{inputenc}
\usepackage{amssymb}
\usepackage{amsfonts}
\usepackage{amscd}
\usepackage{mathrsfs}
\usepackage[dvips]{graphicx}
\usepackage[T1]{fontenc}
\usepackage[english]{babel}
\usepackage{marvosym}
\usepackage{amsthm}
\usepackage{multicol}
\usepackage[all]{xy}

\theoremstyle{plain}
\newtheorem{thm}{Theorem}[section]
\newtheorem{lem}[thm]{Lemma}

\newtheorem{conj}[thm]{Conjecture}
\newtheorem{coro}[thm]{Corollary}

\theoremstyle{definition} \theoremstyle{remark}
\newtheorem{rk}[thm]{Remark}
\newtheorem{df}[thm]{Definition}

\numberwithin{equation}{section}

\def\Pol{\mathcal{P}ol}

\def\ui{{\mathbf{i}}}
\def\uj{{\mathbf{j}}}
\def\uk{{\mathbf{k}}}

\def\bbC{\mathbb{C}}
\def\bbD{\mathbb{D}}

\def\bbN{\mathbb{N}}

\def\bbQ{\mathbb{Q}}

\def\bbZ{\mathbb{Z}}

\def\scrH{\mathscr{H}}

\def\scrL{\mathscr{L}}

\def\frakL{\mathfrak{L}}
\def\frakM{\mathfrak{M}}

\def\frakS{\mathfrak{S}}

\def\frakZ{\mathfrak{Z}}

\def\calA{\mathcal{A}}

\def\calD{\mathcal{D}}
\def\calE{\mathcal{E}}
\def\calF{\mathcal{F}}
\def\calG{\mathcal{G}}
\def\calH{\mathcal{H}}

\def\calL{\mathcal{L}}

\def\calO{\mathcal{O}}
\def\calP{\mathcal{P}}
\def\calQ{\mathcal{Q}}

\def\frakg{\mathfrak{g}}

\def\bff{\mathbf{f}}

\def\bfd{\mathbf{d}}

\def\bfk{\mathbf{k}}

\def\bfv{\mathbf{v}}

\def\bfS{\mathbf{S}}

\def\bfZ{\mathbf{Z}}

\def\bfF{\mathbf{F}}
\def\bfP{\mathbf{P}}

\def\grdim{\operatorname{grdim}}

\def\homo{\operatorname{\it \mathscr{H}\kern-.25em om}}
\def\ext{\operatorname{\it \mathscr{E}\kern-.25em xt}}
\def\edo{\operatorname{\it \mathscr{E}\kern-.25em nd}}
\def\der{\operatorname{\it \mathscr{D}\kern-.25em er}}

\def\Hom{\operatorname{Hom}\nolimits}

\def\End{\operatorname{End}\nolimits}

\def\Ext{\operatorname{Ext}\nolimits}
\def\Res{\operatorname{Res}\nolimits}
\def\Ind{\operatorname{Ind}\nolimits}
\def\Lie{\operatorname{Lie}\nolimits}

\def\Ker{\operatorname{Ker}\nolimits}

\def\Gr{\operatorname{Gr}\nolimits}

\def\Id{\operatorname{Id}\nolimits}

\author{Ruslan Maksimau}
\address{Universit\' e de Paris 7, 175 rue du Chevaleret, 75013 Paris, France}
\email{ruslmax@gmail.com}

\title
[Canonical basis, KLR-algebras and parity sheaves] {Canonical basis,
KLR-algebras and parity sheaves}

\begin{document}
\begin{abstract}
We give a construction of a basis of the positive part of the
Drinfeld-Jimbo quantum enveloping algebra associated with a Dynkin
quiver in terms of parity sheaves.
\end{abstract}

\maketitle

\setcounter{tocdepth}{2}

\tableofcontents

\section{Introduction}
By a variety we will always mean a reduced and quasi-projective
scheme of finite type over $\mathbb{C}$.

Let $\mathbf{k}$ be a field. Let $\Gamma$ be a Dynkin quiver with
the set of vertices $I$ and let $\nu$ be a dimension vector. Let
$E_V$ be the space of representations of $\Gamma$ on a
$\nu$-dimensional $I$-graded vector space $V$. Let $G_V$ be the
group of graded linear automorphisms of $V$. Let $Y_\nu$ be the set
of all pairs $y=(\ui,\mathbf{a})$ where $\ui=(i_1,i_2,\cdots,i_m)$
is a sequence of elements of $I$ and
$\mathbf{a}=(a_1,a_2,\cdots,a_m)$ is a sequence of positive integers
such that $\sum_{l=1}^ma_li_l=\nu$. Let $I^\nu\subset Y_\nu$ be the
set of all pairs $y=(\ui,\mathbf{a})$ such that $a_l=1$ for all l.
We will abbreviate $\ui$ for $y=(\ui,\mathbf{a})$ if $y\in I^\nu$.
For each $y\in Y_\nu$ let $\pi_{y}\colon\widetilde{F}_{y}\to E_V$ be
the Lusztig analogue of the Springer map. Let
$\calL_{V,\bfk}=\bigoplus\limits_{\ui\in
I^\nu}\pi_{\ui*}\underline{\mathbf{k}}_{\widetilde{F}_{\ui}}$ be the
Lusztig complex. Set $^\delta\calL_{V,\bfk}=\bigoplus\limits_{\ui\in
I^\nu}\pi_{\ui*}\underline{\mathbf{k}}_{\widetilde{F}_{\ui}}[\dim_\bbC\widetilde{F}_\ui]$.
Let $\bff$ be the positive part of the Drinfeld-Jimbo quantized
enveloping algebra associated with the quiver $\Gamma$. Set
$\calA=\bbZ[q,q^{-1}]$. Denote by $_\calA\bff$ the Lusztig's
integral form of $\bff$.

The KLR-algebras are recently introduced by Khovanov and Lauda
\cite{KL} and by Rouquier \cite{Rouq}. They yield categorification
of quantum groups. Denote by $R_{\nu,\mathbf{k}}$ the KLR-algebra
over $\mathbf{k}$ associated with the quiver $\Gamma$ and the
dimension vector $\nu$. We denote by $D_{G_V}(E_V, \bfk)$ the
bounded $G_V$-equivariant derived category of sheaves of
$\bfk$-vector spaces on $E_V$ constructible with respect to the
stratification of $E_V$ by $G_V$-orbits. Let $\mathcal{Q}_V$ be the
full additive subcategory of $D_{G_V}(E_V, \bfk)$ such that the
objects of $\calQ_V$ are the direct sums of shifts of direct factors
of $\calL_{V,\bfk}$. Denote by $\mathrm{proj}(R_{\nu,\mathbf{k}})$
the category of $\mathbb{Z}$-graded projective finitely generated
modules over $R_{\nu,\mathbf{k}}$.

The goal of Section 2 is to give a geometric construction of
KLR-algebras in arbitrary characteristic. The zero characteristic
case was done in \cite{Rouq2011} and \cite{VV}. The KLR-algebras in
positive characteristic were studied in \cite{KlRam}. More
precisely, we prove the following in Theorems
\ref{thm_isom-grad-KLR-sh}, \ref{thm_cat-equiv-grad}.
\begin{thm}\label{thm_sec2}\text{ }
\begin{enumerate}
    \item There exists a graded $\mathbf{k}$-algebra isomorphism $R_{\nu,\mathbf{k}}=\mathrm{Ext}_{G_V}^*({^\delta\calL_{V,\bfk}},{^\delta\calL_{V,\bfk})}$.
    \item The functor
$ \mathcal{Q}_V\to
\mathrm{proj}(R_{\nu,\mathbf{k}}),~\mathcal{F}\mapsto
\mathrm{Ext}_{G_V}^*(\mathcal{F},{^\delta\calL_{V,\bfk}}) $ yields
an equivalence of categories from the opposite category
$\calQ_V^{\mathrm{op}}$ of $\calQ_V$ to the category
$\mathrm{proj}(R_{\nu,\mathbf{k}})$.
\end{enumerate}
\end{thm}

Let us denote by $K(R_{\nu,\bfk})$ the split Grothendieck group of
the additive category $\mathrm{proj}(R_{\nu,\mathbf{k}})$. Set $
R_\bfk=\bigoplus_{\nu\in \bbN
I}R_{\nu,\mathbf{k}},~K(R_\bfk)=\bigoplus_{\nu\in \bbN
I}K(R_{\nu,\mathbf{k}}). $ Consider also the split Grothendieck
group $K(\calQ_V)$ of the additive category $\calQ_V$ and set $
\calQ=\bigoplus_{V}\calQ_V,~ K(\calQ)=\bigoplus_{V}K(\calQ_V), $
where the sum is taken over the isomorphism classes of $I$-graded
finite dimensional $\bbC$-vector spaces. The $\bbZ$-modules
$K(\calQ)$ and $K(R_\bfk)$ have the $\calA$-modules structures,
where $q$ acts by the shift of grading. The induction and
restriction functors yield $\calA$-algebra and $\calA$-coalgebra
structures on $K(R_\bfk)$. Theorem \ref{thm_sec2} (2) yields an
$\calA$-linear isomorphism $K(\calQ)\to K(R_\bfk)$. So we can
transfer the algebra structure from $K(R_\bfk)$ to $K(\calQ)$. It
happens that this algebra structure coincides with the algebra
structure given by Lusztig's induction and restriction functors, see
Section \ref{subs_new-bas}. This yields an $\calA$-algebra
isomorphism $\lambda_\calA\colon K(\calQ)\to {_\calA\bff}$, see
Theorem \ref{thm_isom-sh-f}. However, there is no geometric
construction of a coproduct on $K(\calQ)$ via the geometry
associated with $\calQ$.

In Section 3 we study parity sheaves on $E_V$. Let
$\mathrm{Par}_{G_V}(E_V)$ be the full additive subcategory of parity
complexes in $D_{G_V}(E_V, \bfk)$. Let $K(\mathrm{Par}_{G_V}(E_V))$
be its split Grothendieck group. We set
$$
\mathrm{Par}=\bigoplus_V\mathrm{Par}_{G_V}(E_V),\qquad
K(\mathrm{Par})=\bigoplus_VK(\mathrm{Par}_{G_V}(E_V)),
$$
where the sum is taken over the isomorphism classes of $I$-graded finite dimensional $\bbC$-vector spaces. The $\bbZ$-module $K(\mathrm{Par})$ has the $\calA$-module structure, where $q$ acts by the shift of grading. The Lusztig's restriction functor $\Res$ yields an $\calA$-coalgebra structure on $K(\mathrm{Par})$. However, a priori we have no geometric construction of an algebra structure on $K(\mathrm{Par})$. 
In Section 3 we get the following result.
\begin{thm}
\label{thm_sec3} There exists an $\calA$-coalgebra isomorphism
$\beta_\calA\colon K(\mathrm{Par})\to {_\calA}\bff$.
\end{thm}
Theorem \ref{thm_sec3} yields an $\calA$-basis in $_\calA\bff$ in
terms of parity sheaves.  We say that the quiver $\Gamma$ is
$\bfk$-\emph{even} if the complexes
$\pi_{y*}\underline{\mathbf{k}}_{\widetilde{F}_{y}}$ are even for
each $y\in Y_\nu$. We say that $\Gamma$ is \emph{even} if it is
$\bfk$-even for each field $\bfk$. If the quiver $\Gamma$ is
$\bfk$-even then the categories $\calQ_V$ and
$\mathrm{Par}_{G_V}(E_V)$ coincide and we have a bialgebra
isomorphism $\lambda_\calA\colon K(\calQ)\to {_\calA}\bff$, see
Section \ref{subs_ev-quiv}.
Note that this is the case if the characteristic of $\bfk$ is zero.
However we have no example of a Dynkin quiver that is not
$\bfk$-even for some field $\bfk$ of positive characteristic. This
suggests the following conjecture.
\begin{conj}
\label{conj_even} Each Dynkin quiver is even.
\end{conj}
To prove Conjecture \ref{conj_even} it is enough to verify that the
fibers of the morphisms $\pi_y$, $y\in Y_\nu$ have no odd cohomology
group over each field, see Lemma \ref{lem_ev-resol}.

Let $\Lambda_V$ be the set of $G_V$-orbits in $E_V$. For $\lambda\in
\Lambda_V$ we will write $\calO_\lambda$ for the corresponding
orbit. The parity sheaves in $D_{G_V}(E_V, \bfk)$ and the
indecomposable objects in $\calQ_V$ modulo shifts are parametrized
by the the set $\Lambda_V$, see Section \ref{subs_indec-QV} and
Remark \ref{rem_indec-compl} (a). We denote them respectively
$\calE(\lambda)$ and $R_\lambda$ with $\lambda\in\Lambda_V$. The
classes of the complexes $R_\lambda$, with $\lambda\in \Lambda_V$,
form an $\calA$-basis in $K(\calQ_V)$. Thus, the elements
$\lambda_\calA([R_\lambda])$ with $\lambda\in\Lambda_V$ and all $V$
form an $\calA$-basis in $_\calA\bff$. On the other hand the classes
of the complexes $\calE(\lambda)$ with $\lambda\in \Lambda_V$ form
an $\calA$-basis in $K(\mathrm{Par}_{G_V}(E_V))$ and thus the
elements $\beta_\calA([\calE(\lambda)]$ with $\lambda\in\Lambda_V$
and all $V$ form an $\calA$-basis in $_\calA\bff$. It is natural to
compare these bases. We expect that the following result holds.
\begin{conj}
\label{conj_bas-Nak=Lu} For each finite dimensional $I$-graded
$\bbC$-vector space $V$ and each $\lambda\in \Lambda_V$ we have $
\lambda_\calA([R_\lambda])=\beta_\calA([\calE(\lambda)]). $
\end{conj}
Note that Conjecture \ref{conj_bas-Nak=Lu} follows directly from
Conjecture \ref{conj_even} because if $\Gamma$ is $\bfk$-even we
have $R_\lambda=\calE(\lambda)$ for each $\lambda\in\Lambda_V$ and
each $V$, see Remark \ref{rem_indec-compl} (b).

\section{Geometric construction of KLR-algebras}
In this paper $A$ will always denote a commutative ring of finite
global dimension.

\subsection{Quivers\label{subs_quiv}}
Let $\Gamma$ be a quiver without loops. We denote by $I$ and $H$ the
sets of its vertices and arrows respectively. For an arrow $h\in H$
we will write $h'$ and $h''$ for its source and target respectively.
For a dimension vector $\nu=\sum_{i\in I}\nu_i\cdot i\in\mathbb NI$
we set
$$
E_V=\bigoplus_{h\in H}{\rm Hom}(V_{h'},V_{h''}),\qquad
|\nu|=\sum_{i\in I}\nu_i,
$$
 where $V_i$ is a $\mathbb C$-vector space of dimension $\nu_i$ for every $i\in I$. If $\Gamma$ is a Dynkin quiver, the natural action of $G_V=\prod_{i\in I}GL(V_i)$ on $E_V$ has finitely many orbits. This defines a stratification $\coprod_{\lambda\in\Lambda}\mathcal{O}_\lambda$ on $E_V$, where $\Lambda_V$ labels all orbits.

Let us introduce some notation for a later use. Let $h_{i,j}$ be the
number of arrows from $i$ to $j$ in $\Gamma$ and set
$$
i\cdot j=-h_{i,j}-h_{j,i},~~i\cdot i=2,~~i\ne j.
$$
Let $Y_\nu$ be the set of all pairs $y=(\ui,\mathbf{a})$ where
$\ui=(i_1,i_2,\cdots,i_m)$ is a sequence of elements of $I$ and
$\mathbf{a}=(a_1,a_2,\cdots,a_m)$ is a sequence of positive integers
such that $\sum_{l=1}^ma_li_l=\nu$. We will write
$y=(i_1^{(a_1)}\cdots i_m^{(a_m)})$. Let $I^\nu\subset Y_\nu$ be the
set of all pairs $y=(\ui,\mathbf{a})$ such that $a_l=1$ for all l.
We will abbreviate $\ui$ for $y=(\ui,\mathbf{a})$ if $y\in I^\nu$.
We denote by $F_y$ the variety of all flags
$$
\phi=(\{0\}=V^0\subset V^1\subset\cdots\subset V^m=V)
$$
in $V=\bigoplus_{i\in I}V_i$ such that the $I$-graded vector space
$V^r/V^{r-1}$ has graded dimension $a_r\cdot i_r$ for $r\in [1,m]$.
We denote by $\widetilde{F}_{y}$ the variety of pairs $(x,\phi)\in
E_V\times F_y$ such that $x(V^r)\subset V^r$ for
$r\in\{0,1,\cdots,m\}$. Let $\pi_{y}$ be the natural projection from
$\widetilde{F}_{y}$ to $E_V$, i.e.,
$ \pi_{y}:\widetilde{F}_y\to E_V,~(x,\phi)\mapsto x. $
For $y_1,y_2\in Y_\nu$ we denote by $Z_{y_1,y_2}$ the variety of
triples $(x,\phi_1,\phi_2)\in E_V\times F_{y_1}\times F_{y_2}$ such
that $x$ preserves $\phi_1$ and $\phi_2$. For $\ui\in I^\nu$ and
$l=1,2,\hdots,|\nu|$, we define
$\mathcal{O}_{\widetilde{F}_{\ui}}(l)$ to be the $G_V$-equivariant
line bundle over $\widetilde{F}_{\ui}$ whose fiber at the flag
$\phi$ is equal to $V^l/V^{l-1}$.

\subsection{KLR-algebra\label{subs_KLR}}
Set $m=|\nu|$. The group $\mathfrak{S}_m$ acts on $I^\nu$ by
$w\cdot(i_1,\hdots,i_m)=(i_{w^{-1}(1)},\hdots,i_{w^{-1}(m)})$. We
denote by $s_l$ the transposition $(l,l+1)\in \frakS_m$ for
$l\in[1,m-1]$. For each sequence $\ui=(i_1,i_2,\hdots,i_m)\in I^\nu$
and each integers $k,l\in[1,m]$, $l\ne m$ we abbreviate
$$
h_{\ui}(l)=h_{i_l,i_{l+1}}\text{ if } s_l(\ui)\ne
\ui,~~h_{\ui}(l)=-1\text{ if } s_l(\ui)=\ui.
$$
$$
a_{\ui}(l)=h_{\ui}(l)+h_{s_l(\ui)}(l)=-i_l\cdot i_{l+1}.
$$

\begin{df}
\label{def_KLR} The \emph{KLR-algebra} $R_{\nu,A}$ (or simply
$R_\nu$ if the ring of coefficients is clear) over the ring $A$
associated with $\Gamma$ and the dimension vector $\nu$ is the
$A$-algebra generated by $1_{\ui}$, $x_{\ui}(k)$, $\tau_{\ui}(l)$
with $\ui\in I^\nu$, $1\le k,l\le m$ and $l\ne m$, modulo the
following defining relations:
\begin{itemize}
    \item $1_{\ui}1_{\ui'}=\delta_{\ui,\ui'}1_{\ui}$,
    \item $\tau_{\ui}(l)=1_{s_l(\ui)}\tau_{\ui}(l)1_{\ui}$,
    \item $x_{\ui}(k)=1_{\ui}x_{\ui}(k)1_{\ui}$,
    \item $x_{\ui}(k)x_{\ui'}(k')=x_{\ui'}(k')x_{\ui}(k)$,
    \item $\tau_{s_l(\ui)}(l)\tau_{\ui}(l)=Q_{i,l}(x_{\ui}(l),x_{\ui}(l+1))$,
    \item $\tau_{s_l(\ui)}(l')\tau_{\ui}(l)=\tau_{s_{l'}(\ui)}(l)\tau_{\ui}(l')\text{ if }|l-l'|>1$,
    \item $\tau_{s_ls_{l+1}(\ui)}(l+1)\tau_{s_{l+1}(\ui)}(l)\tau_{\ui}(l+1)-\tau_{s_{l+1}s_l(\ui)}(l)\tau_{s_l(\ui)}(l+1)\tau_{\ui}(l)$

$=(Q_{\ui,l}(x_{\ui}(l+2),x_{\ui}(l+1))-Q_{\ui,l}(x_{\ui}(l),x_{\ui}(l+1)))(x_{\ui}(l+2)-x_{\ui}(l))^{-1}$
if $s_{l,l+2}(\ui)=\ui$ and 0 else,
    \item $\tau_{\ui}(l)x_{\ui}(k)-x_{s_l(\ui)}(s_l(k))\tau_{\ui}(l)=
\begin{cases}
-1_{\ui} & \text{ if } k=l,~s_l(\ui)=\ui,\\
1_{\ui} & \text{ if } k=l+1,~s_l(\ui)=\ui,\\
0 & \text{ else.}
\end{cases}$
\end{itemize}
Here we have set $s_{l,l+2}=s_ls_{l+1}s_l$ if $l\ne m-1,m$ and
$$
Q_{\ui,l}(u,v)=
\begin{cases}
(-1)^{h_{\ui}(l)}(u-v)^{a_{\ui}(l)} & \text{ if } s_l(\ui)\ne\ui,\\
0 & \text{ else.}
\end{cases}
$$
\end{df}
Note that the fraction
$$
(Q_{\ui,l}(x_{\ui}(l+2),x_{\ui}(l+1))-Q_{\ui,l}(x_{\ui}(l),x_{\ui}(l+1)))(x_{\ui}(l+2)-x_{\ui}(l))^{-1}
$$
is indeed an element of $A[x_{\ui}(1),\cdots x_{\ui}(m)]$. Note also
that $ R_{\nu,A}=\bigoplus_{\ui,\uj\in I^\nu}R_{\nu,A}^{\ui\uj}, $
where $R_{\nu,A}^{\ui\uj}=1_{\ui}R_{\nu,A}1_{\uj}$.

The KLR-algebra $R_{\nu,A}$ has a natural $\bbZ$-grading such that
$\deg 1_{\ui}=0$, $\deg x_{\ui}(k)=2$ and $\deg
\tau_{\ui}(l)=a_{\ui}(l)$. The relations in Definition \ref{def_KLR}
are homogeneous with respect to this grading.

\subsection{A faithful representation of $R_{\nu,A}$\label{subs_f-rep}}
Now, we consider the representation of $R_{\nu,A}$ on the space
$$
\Pol_{\nu,A}=\bigoplus_{\ui\in
I^{\nu}}\Pol_{\ui,A},\qquad\Pol_{\ui,A}=A[x_{\ui}(1),\cdots,x_{\ui}(m)],
$$
which is given by:
\begin{itemize}
    \item the element $1_{\ui}\in R_{\nu,A}$ acts on $\Pol_{\nu,A}$ by the projection on $\Pol_{\ui,A}$,
    \item the element $x_{\ui}(k)\in R_{\nu,A}$ acts on $\Pol_{\uj,A}$ by $0$ if $\ui\ne\uj$ and by multiplication by $x_k(\ui)$ if $\ui=\uj$,
   \item the element $\tau_{\ui}(l)$ acts on $\Pol_{\uj,A}$ by $0$ if $\ui\ne\uj$ and it sends $f\in\Pol_{\ui,A}$ to
\begin{itemize}
\item[\textbullet] $-\frac{f-s_l(f)}{x_{\ui}(l)-x_{\ui}(l+1)}$ if $s_l(\ui)=\ui$,
\item[\textbullet] $(x_{s_l(\ui)}(l)-x_{s_l(\ui)}(l+1))^{h_{\ui}(l)}s_l(f)$ if $s_l(\ui)\ne \ui$.
\end{itemize}
\end{itemize}
This representation is faithful, see \cite[Sec.~3.2.2]{Rouq}

\subsection{Reminder on Borel-Moore homology\label{subs_BM}}
In this section we suppose $G$ is an algebraic group which acts on a variety $X$. 
Denote by $D_G(X,A)$ the $G$-equivariant derived category of
sheaves of $A$-modules on $X$. Let $\mathcal{D}_{X,A}$ be the $G$-equivariant
dualizing complex on $X$, see \cite[Def.~3.1.16]{KSch} in the
non-equivariant case and \cite[Def.~3.5.1]{BL} in the
equivariant case.
The space of $G$-\emph{equivariant Borel-Moore homology} is
$\mathrm{H}_*^G(X,A)=\mathrm{H}_G^*(X,\mathcal{D}_{X,A})$.
Note also that $\mathcal{D}_{X,A}=\underline{A}_X[2d]$ if $X$ is a
smooth $G$-variety of pure dimension $d$, where we denote by
$\underline{A}_X$ the $G$-equivariant constant sheaf on $X$, see
\cite[Prop.~8.3.4]{CG}. In this case we have
$$
\mathrm{H}_*^G(X,A)=\mathrm{H}_G^*(X,\underline{A}_X[2d]). \eqno
(2.1)
$$
So we can identify $\mathrm{H}_*^G(X,A)$ with
$\mathrm{H}_G^*(X,A)[2d]$. The functor
$$
\mathbb{D}=\mathbb{D}_X: D_G(X,A)\to D_G(X,A),\qquad
\mathfrak{L}\mapsto \scrH{o m}(\mathfrak{L},\mathcal{D}_{X,A})
$$
has the following properties, see \cite[Thm.~3.5.2]{BL}:
\begin{itemize}
    \item $(\mathbb{D}_X)^2=\mathrm{Id}_{D_G(X,A)}$,
    \item $\mathbb{D}_Yf_!=f_*\mathbb{D}_X,~\forall f:X\to Y$ continuous $G$-map,
    \item $\mathbb{D}_Xf^*=f^!\mathbb{D}_Y,~\forall f:X\to Y$ continuous $G$-map.
\end{itemize}

\subsection{Algebra $\mathbf{Z}_{V,A}$\label{subs_alg-Z}}
Set $G=GL(V)$. Then $G_V$ is a closed reductive subgroup of $G$.
Consider the varieties $Z_V=\coprod\limits_{\ui,\uj\in
I^\nu}Z_{\ui,\uj}$ and $\widetilde{F}_V=\coprod\limits_{\ui\in
I^\nu}\widetilde{F}_{\ui}$, see Section \ref{subs_quiv}. For
$\ui,\uj\in I^\nu$ we set
$$
\mathbf{Z}_{V,A}=\mathrm{H}_*^{G_V}(Z_V,A),\qquad
\mathbf{Z}_{\ui,\uj,A}=\mathrm{H}_*^{G_V}(Z_{\ui,\uj},A),
$$
$$
\mathbf{F}_{V,A}=\mathrm{H}_*^{G_V}(\widetilde{F}_V,A),\qquad
\mathbf{F}_{\ui,A}=\mathrm{H}_*^{G_V}(\widetilde{F}_{\ui},A).
$$
We equip $\mathbf{Z}_{V,A}$ with the convolution product $\star$
relative to the closed embedding $Z_V\subset \widetilde{F}_V\times
\widetilde{F}_V$. See \cite[Sec.~2.7]{CG} for details. In the
same way we have the $\mathbf{Z}_{V,A}$-module structure on
$\mathbf{F}_{V,A}$. We will denote it also by $\star$.

Denote by $F$ the variety of all complete flags in $V$. Denote by
$F_V$ the variety of all complete flags in $V$ that agree with the
grading of $V$, i.e., $F_V=\coprod\limits_{\ui\in I^\nu}F_{\ui}$.
Then $F_V\subset F$. Let $T_V$ be a maximal torus in $G_V$. Denote
by $\mathfrak{S}_V$ and $\mathfrak{S}$ the Weyl groups of
$(G_V,T_V)$ and $(G,T_V)$ respectively. Fix once and for all a
$T_V$-stable flag $\phi_V$ in $F_V$. Write
$$
\phi_V=(0\subset D_1\subset D_1\oplus D_2\subset\cdots\subset
D_1\oplus D_2\oplus\cdots\oplus D_{m-1}\subset D_1\oplus
D_2\oplus\cdots\oplus D_m=V),
$$
where $D_1,\cdots,D_m$ are $T_V$-stable one-dimensional vector
spaces in $V$. The action of $\mathfrak{S}$ on the set
$\{D_1,D_2,\cdots,D_m\}$ identifies $\mathfrak{S}$ with the
symmetric group $\mathfrak{S}_m$. Note also that the action of
$\mathfrak{S}$ on the set $\{D_1,D_2,\cdots,D_m\}$ induces the
action of $\mathfrak{S}$ on the set of $T_V$-stable flags in $F_V$.
Set $\phi_{V,w}=w(\phi_V)$ for $w\in\mathfrak{S}$. Denote by $B$ and
$B_V$ the stabilizers of $\phi_V$ in $G$ and $G_V$ respectively. Let
$\Delta$ and $\Delta_V$ be the root systems of $(G,T_V)$ and
$(G_V,T_V)$ respectively. Let $\Delta^+$ be the subset of positive
root in $\Delta$ associated with $B$. We set
$$
\Delta^-=-\Delta^+,\qquad \Delta_V^+=\Delta^+\cap \Delta_V,\qquad
\Delta_V^-=\Delta^-\cap \Delta_V.
$$
Let $\Pi\subset \mathfrak{S}$ be the set of simple reflections. Let
$\chi_1,\chi_2,\cdots,\chi_m\in\mathfrak{t}_V^*$ be the weights of
the lines $D_1,D_2,\cdots,D_m$ respectively, where
$\mathfrak{t}_V=\mathrm{Lie}(T_V)$. The set of simple roots in
$\Delta^+$ is $ \{\chi_l-\chi_{l+1};~l=1,2,\hdots,m-1\}. $ Let
$s_l\in\Pi$ denote the reflection with respect to the simple root
$\chi_l-\chi_{l+1}$. Under the identification
$\mathfrak{S}=\mathfrak{S}_m$ the reflection $s_l$ is a
transposition $(l,l+1)$.

\subsection{Algebras $\mathbf{P}_{V,A}$ and $\mathbf{S}_{V,A}$\label{subs_alg-S-P}}
Denote by $\mathbf{P}_{V,A}$ the $A$-algebra
$\mathrm{H}_{T_V}^*(\bullet,A)$. The A-algebra $\mathbf{P}_{V,A}$ is
isomorphic to the algebra of polynomials on $\mathfrak{t}_V$:
$$
\mathbf{P}_{V,A}=A[\chi_1,\chi_2,\hdots,\chi_m],
$$
where $\chi_1,\hdots,\chi_m$ are as in Section \ref{subs_alg-Z}. Let
$\mathbf{S}_{V,A}$ be the $A$-algebra
$\mathrm{H}_{G_V}^*(\bullet,A)$. For each positive integer $n$ the
fundamental group of $GL_n(\mathbb{C})$ is equal to $\mathbb{Z}$.
Thus, the fundamental group of $G_V$ does not contain elements of
finite order. Hence, by \cite[Lem.~1]{Totaro} the torsion index of
$G_V$ is equal to $1$ (see \cite[Sec.~2]{Tor} for the definition
of a torsion index of a Lie group). So by \cite[Cor.~2.3]{Tor}
the algebra $\mathbf{S}_{V,A}$ is isomorphic to the algebra
$\mathbf{P}_{V,A}^{\mathfrak{S}_V}$ of $\mathfrak{S}_V$-invariant
polynomials on $\mathfrak{t}_V$.

\subsection{Stratification on $F_V\times F_V$\label{subs_str-FV}}
The group $G$ acts diagonally on $F\times F$. The action of the
subgroup $G_V$ preserves the subset $F_V\times F_V$. For $x\in
\mathfrak{S}$ let $O_V^x$ be the set of all pairs of flags in
$F_V\times F_V$ which are in relative position $x$. More precisely,
write $ \phi_{V,w',w}=(\phi_{V,w'},\phi_{V,w}),~\forall
w,w'\in\mathfrak{S}. $ Then we set $ O_V^x=(F_V\times F_V)\cap
G\phi_{V,e,x}. $

\subsection{Stratification on $Z_{V}$\label{subs_str-Z}}

Let $Z_V^x$ be the Zariski closure of the locally closed subset
$q^{-1}(O_V^x)$,  where $q$ is a natural map
$$
q:Z_V\to F_V\times F_V,~(x,\mathcal{F}_1,\mathcal{F}_2)\mapsto
(\mathcal{F}_1,\mathcal{F}_2).
$$
Put
$$
Z_V^{\leqslant x}=\bigcup_{w\leqslant
x}Z_V^w,\qquad\mathbf{Z}_{V,A}^{\leqslant
x}=\mathrm{H}_*^{G_V}(Z_V^{\leqslant
x},A),\qquad\mathbf{Z}_{V,A}^e=\mathrm{H}_*^{G_V}(Z_V^{\leqslant
e},A).
$$
$$
Z_V^{<x}=\bigcup_{w<x}Z_V^w,\qquad\mathbf{Z}_{V,A}^{<x}=\mathrm{H}_*^{G_V}(Z_V^{<x},A),
$$
where $\leqslant$ is the Bruhat order.

\subsection{Generators of $\mathbf{Z}_{V,A}$\label{subs_gen-Z}}
Denote by $1_{\ui}$ the fundamental class of $Z_{\ui,\ui}$ in
$\mathrm{H}_*^{G_V}(Z_{\ui,\ui},A)$ regarded as an element of
$\mathrm{H}_*^{G_V}(Z_V,A)$. Fix a simple reflection $s=s_l$ in
$\Pi$. The fundamental class of $Z_V^s$ yields an elements
$\sigma(l)\in \mathbf{Z}_{V,A}^{\leqslant s}$. For $\ui, \ui'\in
I^\nu$ we write
$$
\sigma_{\ui',\ui}(l)=1_{\ui'}\star \sigma(l)\star
1_{\ui}\in\mathbf{Z}_{V,A}^{\leqslant s}.
$$
Next, suppose $k\in[1,m]$. The pull-back of the first Chern class of
the line bundle
$\bigoplus\limits_{\ui}\mathcal{O}_{\widetilde{F}_{\ui}}(k)$ by the
first projection
$$
\begin{CD}
{Z_V^e\subset \widetilde{F}_V\times\widetilde{F}_V} @>pr_1>>
{\widetilde{F}_V}
\end{CD}
$$
belongs to $\mathrm{H}_{G_V}^*(Z_V^e,A)$, here
$\mathcal{O}_{\widetilde{F}_{\ui}}(k)$ is as in Section
\ref{subs_quiv}. It yields an element $x(k)\in\mathbf{Z}_{V,A}^e$.
We write
$$
x_{\ui',\ui}(k)=1_{\ui'}\star x(k)\star 1_{\ui}\in
\mathbf{Z}_{V,A}^e.
$$

We will need the following standard lemma, see \cite[Chap.~6]{CG}.
\begin{lem}
\label{lem_cl-emb-Z}
\begin{itemize}
    \item[(a)] The direct image by the closed embedding $Z_V^{\leqslant x}\subset Z_V$ gives an injective $A$-module homomorphism $\mathbf{Z}_{V,A}^{\leqslant x}\to\mathbf{Z}_{V,A}$.
    \item[(b)] For $x,y\in \mathfrak{S}$ such that $l(xy)=l(x)+l(y)$ the convolution product in $\mathbf{Z}_{V,A}$ yields an inclusion $\mathbf{Z}_{V,A}^{\leqslant x}\star\mathbf{Z}_{V,A}^{\leqslant y}\subset \mathbf{Z}_{V,A}^{\leqslant xy}$. In particular $\mathbf{Z}_{V,A}^{e}$ is a subalgebra of $\mathbf{Z}_{V,A}$.
\end{itemize}
\end{lem}
\begin{proof}
All the statements are obvious except, perhaps, the injectivity in
$(a)$. It follows from the existence of a decomposition
$$
Z_V^{\leqslant x}=\coprod_{w\leqslant x}q^{-1}(O_V^w).
$$
We need the following obvious lemma.
\begin{lem}
\label{lem_auxil-fibr} For each $G_V$-orbit $\calO$ in $F_V\times
F_V$ the map
$p:q^{-1}(\calO) \to F_\ui$ such that $(x,\phi_1,\phi_2)\mapsto \phi_2$ for each $x\in E_V,~\phi_1,\phi_2\in F_V$
is an affine $G_V$-equivariant fibration, where $\ui$ is the unique
element of $I^\nu$ such that $\calO\subset F_V\times F_\ui$.
\end{lem}

Now we can conclude applying the arguments of \cite[Sec.~6.2]{CG}
and the exact sequence in equivariant Borel-Moore homology. We
should just replace \cite[Lem.~6.2.5]{CG} by Lemma
\ref{lem_auxil-fibr}.
\end{proof}

So we can consider $x_{\ui',\ui}(k)$ and $\sigma_{\ui',\ui}(l)$ as
elements of $\mathbf{Z}_{V,A}$.

\subsection{The $A$-algebra structure on $\mathbf{F}_{V,A}$\label{subs_alg-F}}
Let $\ui\in I^\nu$. Assigning to a formal variable $x_{\ui}(l)$ of
degree 2 the first Chern class of $G_V$-equivariant line bundle
$\mathcal{O}_{\widetilde{F}_{\ui}}(l)$ for $l=1,2,\hdots,m$ yields a
graded $A$-algebra isomorphism
$$
A[x_{\ui}(1),x_{\ui}(2),\hdots,x_{\ui}(m)]=\mathbf{F}_{\ui,A}.
$$
The multiplication of polynomials equips each $\mathbf{F}_{\ui,A}$
with an obvious graded $A$-algebra structure. Thus,
$\mathbf{F}_{V,A}=\bigoplus_{\ui\in I^\nu}\bfF_{\ui,A}$ is also a
graded $A$-algebra.

Now we can consider $\mathbf{Z}_{V,A}$ as a
$\mathbf{F}_{V,A}$-module such that
$f(x_{\ui}(1),\hdots,x_{\ui}(m))\in \mathbf{F}_{\ui,A}$ acts on
$\mathbf{Z}_{V,A}$ by the operator
$z\mapsto f(x_{\ui}(1),\hdots,x_{\ui}(m))\star z$.
We can interpret this $\mathbf{F}_{V,A}$-action on
$\mathbf{Z}_{V,A}$ as follows. Note that

\begin{eqnarray*}
\widetilde{F}_V & = & \{(x,\phi)\in E_V\times F_V;~x\text{ preserves }\phi\}=\\
& = & \{(x,\phi,\phi)\in E_V\times F_V\times F_V;~x\text{ preserves }\phi\}=\\
& = & Z_V^e.
\end{eqnarray*}

Then $\mathbf{F}_{V,A}=\mathbf{Z}_{V,A}^e$. So the
$\mathbf{F}_{V,A}$-action on $\mathbf{Z}_{V,A}$ is the action of the
subalgebra $\mathbf{Z}_{V,A}^e$ on $\mathbf{Z}_{V,A}$.

\subsection{KLR-algebra and Borel-Moore homology\label{subs_KLR-BM}}

For each sequence $\ui=(i_1,i_2,\hdots,i_m)\in I^\nu$ and for each
integers $k,l\in[1,m]$, $l\ne m$ we consider the following elements
$ x_{\ui}(k)=x_{\ui,\ui}(k),~\sigma_\ui(l)=\sigma_{s_l(\ui),\ui}(l)
$ in $\bfZ_{V,A}$.

\begin{thm}
\label{thm_isom-KLR-Z} We have an $A$-algebra isomorphism $
R_{\nu,A}\simeq\mathbf{Z}_{V,A}. $
\end{thm}
\begin{proof}
First, we claim that there is an $A$-algebra homomorphism $
\Phi_A:R_{\nu,A}\to\mathbf{Z}_{V,A} $ given by
$$
\Phi_A(1_{\ui})=1_{\ui},\qquad
\Phi_A(x_{\ui}(k))=x_{\ui}(k),\qquad\Phi_A(\tau_{\ui}(l))=(-1)^{h_{\ui}(l)}\sigma_{\ui}(l).
$$

To prove this we must check that the relations from Definition \ref{def_KLR} hold for the elements  $1_{\ui}$, $x_{\ui}(k)$, $(-1)^{h_{\ui}(l)}\sigma_{\ui}(l)$ in $\mathbf{Z}_{V,A}$. Now according to \cite[Thm.~3.6]{VV}, this holds for $A=\mathbb{C}$. We want to deduce from this that the relations hold for any $A$. To do this, we first construct an $\mathbf{F}_{V,A}$-basis of $\mathbf{Z}_{V,A}$ using the following lemma, which can be proved in the same way as Lemma \ref{lem_cl-emb-Z}. 

\begin{lem}
\label{lem_free-F-mod} We have $\mathbf{Z}_{V,A}^{\leqslant
x}=\bigoplus\limits_{w\leqslant x}\mathbf{F}_{V,A}\star[Z_V^w]$ and
$\mathbf{Z}_{V,A}^{<x}=\bigoplus\limits_{w<x}\mathbf{F}_{V,A}\star[Z_V^w]$
for $w\in\mathfrak{S}$. In particular $\mathbf{Z}_{V,A}$ is a free
graded $\mathbf{F}_{V,A}$-module of rank $m!$.
\end{lem}

By Lemma \ref{lem_free-F-mod} the $\mathbb{Z}$-module
$\mathbf{Z}_{V,\mathbb{Z}}$ is free. So by the universal coefficient
theorem we have an inclusion of rings $
j_A\colon\mathbf{Z}_{V,\mathbb{Z}}\to\mathbf{Z}_{V,A}. $ By
\cite[Thm.~2.5]{KL} the $\mathbb{Z}$-module $R_{\nu,\bbZ}$ is
free. So we have an inclusion $ i_A\colon R_{\nu,\mathbb{Z}}\to
R_{\nu,A}. $ Since $j_\mathbb{C}$ takes the elements $1_{\ui}$,
$x_{\ui}(k)$, $(-1)^{h_{\ui}(l)}\sigma_{\ui}(l)$ in
$\mathbf{Z}_{V,\bbZ}$ to the corresponding elements in
$\mathbf{Z}_{V,\bbC}$, the defining relations in $R_{\nu,\bbZ}$ are
satisfied by the generators of $\mathbf{Z}_{V,\bbZ}$. Thus, the ring
homomorphism
$\Phi_{\mathbb{Z}}:R_{\nu,\mathbb{Z}}\to\mathbf{Z}_{V,\mathbb{Z}}$
is well-defined and we have the following commutative diagram
$$
\begin{CD}
R_{\nu,\mathbb{C}} @>\Phi_{\mathbb{C}}>> \mathbf{Z_{V,\mathbb{C}}}\\
@AAi_\mathbb{C}A                       @AAj_\mathbb{C}A\\
R_{\nu,\mathbb{Z}} @>\Phi_{\mathbb{Z}}>> \mathbf{Z_{V,\mathbb{Z}}}.
\end{CD}
$$

Now we must prove that $\Phi_{\mathbb{Z}}$ is bijective. Note that
the map $\Phi_\mathbb{Z}$ is $\mathbf{F}_{V,\mathbb{Z}}$-linear.
Further, the morphism $\Phi_\mathbb{Z}$ is injective because
$i_\mathbb{C}$ and $\Phi_\mathbb{C}$ are. Let us prove that
$\Phi_{\bbZ}$ is surjective. We will need the following lemma.
\begin{lem}
\label{lem_rel-sigma} The following relations hold in
$\mathbf{Z}_{V,\mathbb{Z}}$:
\begin{itemize}
    \item[(a)] $\sigma_{\ui,\uj}(l)=0$ unless $\ui\in\{\uj,s_l(\uj)\}$,
\item[(b)]$\sigma_{\ui,\ui}(l)=(x_{\ui}(l+1)-x_{\ui}(l))^{h_{i_l,i_{l+1}}}$  if $s_l(\ui)\ne \ui$.
\end{itemize}
\end{lem}
\begin{proof}
By \cite[Prop.~2.22]{VV} and \cite[Lem.~1.8~(a)]{VV} these
relations hold in $\mathbf{Z}_{V,\mathbb{C}}$. Thus, they also hold
in $\mathbf{Z}_{V,\mathbb{Z}}$ by injectivity of the map
$j_{\mathbb{C}}$.
\end{proof}
The image of $\Phi_{\mathbb{Z}}$ contains the elements $x_{\ui}(k)$
and $1_{\ui}$ for $\ui\in I^\nu$ and $k\in[1,m]$. So it contains
$\mathbf{F}_{V,\mathbb{Z}}\star [Z_V^e]$ because
$$
[Z_V^e]=\sum\limits_{\ui\in I^\nu} 1_{\ui},
\qquad\mathbf{F}_{V,\mathbb{Z}}=\bigoplus_{\ui\in
I^\nu}\mathbb{Z}[x_{\ui}(1),x_{\ui}(2),\hdots,x_{\ui}(m)].
$$
Next, we prove that the image of $\Phi_\mathbb{Z}$ contains
$[Z_V^{s_l}]$ for $l\in[1,m-1]$. By Lemma \ref{lem_rel-sigma} (a) we
have
\begin{eqnarray*}
[Z_V^{s_l}] & = & \sigma(l)\\
& = &\sum_{\ui,\uj\in I^\nu}\sigma_{\ui,\uj}(l)\\
& = &\sum_{\ui\in I^\nu}\sigma_{s_l(\ui),\ui}(l)+\sum_{\ui\in I^\nu,~s_l(\ui)\ne\ui}\sigma_{\ui,\ui}(l)\\
& = & \sum_{\ui\in I^\nu}\sigma_{\ui}(l)+\sum_{\ui\in
I^\nu,~s_l(\ui)\ne\ui}\sigma_{\ui,\ui}(l).
\end{eqnarray*}
We know that
$\sigma_{\ui}(l)=\Phi_{\mathbb{Z}}((-1)^{h_{\ui}(l)}\tau_{\ui}(l))$.
If $s_l(\ui)\ne \ui$ then by Lemma \ref{lem_rel-sigma} (b) we have
$$
\sigma_{\ui,\ui}(l)=(x_{\ui}(l+1)-x_{\ui}(l))^{h_{i_l,i_{l+1}}}=\Phi_{\mathbb{Z}}((x_{\ui}(l+1)-x_{\ui}(l))^{h_{i_l,i_{l+1}}}).
$$
So $[Z_V^{s_l}]$ is in the image of $\Phi_{\mathbb{Z}}$.
To conclude we use Lemma \ref{lem_free-F-mod} and the following
lemma, see \cite[Lem.~3.7]{VV}.
\begin{lem}
\label{lem_mult-F-basis} If $l(s_lw)=l(w)+1$ we have
$[Z_V^{s_l}]\star[Z_V^w]=[Z_V^{s_lw}]$ in $\mathbf{Z}_V^{\le
s_lw}/\mathbf{Z}_V^{<s_lw}$.
\end{lem}

This proves that $\Phi_{\bbZ}$ is bijective. Thus $\Phi_A$ is
well-defined and bijective for arbitrary commutative ring $A$
because
$R_{\nu,A}=R_{\nu,\mathbb{Z}}\otimes
A$ and $\mathbf{Z}_{V,A}=\mathbf{Z}_{V,\mathbb{Z}}\otimes A$.
\end{proof}

\subsection{Stratified varieties\label{subs_strat-var}}
Let $X$ be a $G$-variety for some connected algebraic group $G$. We
fix a stratification $ X=\coprod_{\lambda\in\Lambda}X_\lambda $ of
$X$ into smooth connected locally closed $G$-stable subsets such
that the closure of each stratum is a union of strata. Let
$D_{G}(X,A)$ be the bounded $G$-equivariant derived category of
sheaves of $A$-modules on $X$ constructible with respect to the
stratification $ X=\coprod_{\lambda\in\Lambda}X_\lambda. $

\begin{lem}
\label{lem_Kr-Sch} Suppose that $A$ is either a complete discrete
valuation ring or a field. The category $D_G(X,A)$ is a
Krull-Schmidt category.
\end{lem}
\begin{proof}
By \cite[Thm.]{Kar} the category $D_G(X,A)$ is Karoubian because
its $t$-structure is bounded, see \cite[Def.~2.2.4,
Prop.~2.5.2]{BL}. Let $\bfk$ be the factor of $A$ by its
unique maximal ideal. We need to prove that the endomorphism ring
$\End(\calF)$ of each indecomposable object $\calF$ in $D_G(X,A)$ is
local. Note that under the assumption on the ring $A$ the
endomorphism ring $\End(\calF)$ is local if and only if the ring
$\bfk\otimes_A \End(\calF)$ is local.

Now, suppose that the ring $\bfk\otimes_A \End(\calF)$ is not local.
The endomorphism ring of every object of $D_G(X,A)$ is of finite
type over $A$, see the construction of the bounded equivariant
derived category in \cite[Sec.~2.1-2.2]{BL}. Thus the ring
$\bfk\otimes_A \End(\calF)$ is a finite dimensional $\bfk$-algebra.
In particular the ring $\bfk\otimes_A \End(\calF)$ is artinian. Then
the identity of $\bfk\otimes_A \End(\calF)$ can be written as a sum
of two orthogonal idempotents $e_1$ and $e_2$ because $\bfk\otimes_A
\End(\calF)$ is not local. Moreover, by \cite[Ex.~6.16]{CR} we
can lift these idempotents to the idempotents $\widetilde e_1$,
$\widetilde e_2$ of $\End(\calF)$ such that $\widetilde e_1 +
\widetilde e_2=\Id_\calF$. Thus by Karoubianness the object $\calF$
is not indecomposable.
\end{proof}

For each $\lambda\in\Lambda$ we denote by $i_\lambda$ the inclusion
$X_\lambda\to X$ and by $d_\lambda$ the complex dimension of
$X_\lambda$. For each $x\in X$ we denote by $i_x$ the inclusion
$\{x\}\to X$.
For all $\lambda\in\Lambda$, let $\underline{A}_\lambda$ be the
$G$-equivariant constant sheaf on $X_\lambda$.
For each $\lambda$ let ${\rm Loc}_{G}(X_\lambda,A)$ be the category
of $G$-equivariant local systems of free $A$-modules on $X_\lambda$.
Let $\underline{A}_X$ be the $G$-equivariant constant sheaf on $X$. 

\subsection{Lusztig category $\calQ_V$\label{subs_QV}}
From now on we suppose that $\Gamma$ is a Dynkin quiver. We have the
stratification $
E_V=\coprod_{\lambda\in\Lambda_V}\mathcal{O}_\lambda $ on $E_V$.
Here $\Lambda_V$ is the set of $G_V$-orbits in $E_V$. For
$\lambda\in \Lambda_V$ we will use the notation $i_\lambda$ and
$d_\lambda$ as in Section \ref{subs_strat-var}, i.e., the map
$i_\lambda$ is the inclusion $\calO_\lambda\to E_V$ and $d_\lambda$
is the complex dimension of $\calO_\lambda$. Consider the following
complex in $D_{G_V}(E_V,A)$:
$$
\mathcal{L}_{V,A}=\bigoplus_{\ui\in I^\nu}\mathcal{L}_{\ui},\qquad
\mathcal{L}_{\ui}=\pi_{\ui*}\underline{A}_{\widetilde{F}_{\ui}}, ~~
\ui\in I^\nu.
$$

Consider also the complex in $D_{G_V}(E_V,A)$:
$$
^\delta\mathcal{L}_{V,A}=\bigoplus_{\ui\in
I^\nu}{^\delta\mathcal{L}_{\ui}},\qquad
^\delta\mathcal{L}_{\ui}=\calL_\ui[d_\lambda], ~~ \ui\in I^\nu.
$$

\begin{df}
Let $\mathcal{Q}_V$ be the full additive subcategory of
$D_{G_V}(E_V,A)$ whose objects are the finite direct sums of shifts of direct factors of $\calL_{V,A}$.
\end{df}

\subsection{Local systems}
The following lemma is proved in \cite[Prop.~2.2.1]{Br}.
\begin{lem}
\label{lem_stab-GL} The reductive part of the stabilizer in $G_V$ of
an element in $E_V$ is isomorphic to
$\prod\limits_{i}GL_{n_i}(\mathbb{C})$ for some positive integers
$n_i$.
\end{lem}
We get the following result from Lemma \ref{lem_stab-GL} and
\cite[Lem.~8.4.11]{CG}.
\begin{coro}
\label{coro_loc-sys-triv} Let $\calO$ be a $G_V$-orbit in $E_V$.
Each local system in $\mathrm{Loc}_{G_V}(\calO)$ is a multiple of
the trivial $G_V$-equivariant local system $\underline A_\calO$ on
$\calO$.
\end{coro}

\subsection{KLR-algebra and Yoneda algebras\label{subs_KLR-Yon}}
Consider the $G_V$-equivariant extension group ${\rm
Ext}_{G_V}^*(\calL_{V,A},\calL_{V,A})$ with the $A$-algebra
structure given by the Yoneda product.
\begin{thm}
\label{thm_isom-KLR-Ext} We have the following $A$-algebra
isomorphism
$R_{\nu,A}={\rm Ext}_{G_V}^*(\calL_{V,A},\calL_{V,A})$.
\end{thm}
\begin{proof}
We have proved in Theorem \ref{thm_isom-KLR-Z} that there is an
$A$-algebra isomorphism
$$
R_{\nu,A}=\mathrm{H}_*^{G_V}(Z_V,A).
$$
Now it suffice to show that there exists as $A$-algebra isomorphism
$$
\mathrm{H}^{G_V}_*(Z_V,A)=\Ext_{G_V}^*(\calL_{V,A},\calL_{V,A}).
$$
For each $\ui,\uj\in I^\nu$ consider the following commutative
diagram
$$
\begin{CD}
Z_{\ui\uj} @>p_2>>\widetilde F_\uj\\
@V{p_1}VV         @V{\pi_\uj}VV\\
F_\ui      @>\pi_\ui>> E_V,
\end{CD}
$$
where $p_1$ and $p_2$ are the natural projections. For each
$n\in\bbN$ we have
\begin{eqnarray*}
\Ext_{G_V}^n((\pi_\ui)_*\underline A_{\widetilde F_\ui},(\pi_\uj)_*\underline A_{\widetilde F_\uj}) & = & \Ext_{G_V}^n(\underline A_{\widetilde F_\ui},(\pi_\ui)^!(\pi_\uj)_*\underline A_{\widetilde F_\uj})\\
& = & \Ext_{G_V}^n(\underline A_{\widetilde F_\ui},(p_1)_*(p_2)^!\underline A_{\widetilde F_\uj})\\
& = & \Ext_{G_V}^n((p_1)^*\underline A_{\widetilde F_\ui},(p_2)^!\underline A_{\widetilde F_\uj})\\
& = & \Ext_{G_V}^n(\underline A_{Z_{\ui\uj}},\calD_{Z_{\ui\uj},A}[-2\dim_{\bbC} \widetilde F_\uj])\\
& = & \mathrm{H}_{G_V}^{n-2\dim_{\bbC} \widetilde F_\uj}(Z_{\ui\uj},\calD_{Z_{\ui\uj},A})\\
& = & \mathrm{H}^{G_V}_{2\dim_{\bbC} \widetilde F_\uj-n}(Z_{\ui\uj},A).\\
\end{eqnarray*}
Thus we have a set bijection between $\mathrm{H}^{G_V}_*(Z_V,A)$ and
$\Ext_{G_V}^*(\calL_{V,A},\calL_{V,A})$. Now we need to show that
the convolution product on $\mathrm{H}^{G_V}_*(Z_V,A)$ coincides
with the Yoneda product on $\Ext_{G_V}^*(\calL_{V,A},\calL_{V,A})$.
At first note that this follows from \cite[Prop.~8.6.35]{CG}
for $A=\bbC$. Now we show this for $A=\bbZ$. We have a commutative
diagram
$$
\begin{CD}
\mathrm{H}^{G_V}_*(Z_V,\bbC) @>>> \Ext_{G_V}^*(\calL_{V,\bbC},\calL_{V,\bbC})\\
@AAA            @AAA\\
\mathrm{H}^{G_V}_*(Z_V,\bbZ) @>>> \Ext_{G_V}^*(\calL_{V,\bbZ},\calL_{V,\bbZ}),\\
\end{CD}
$$
where the left vertical map is an algebra inclusion, the right
vertical map is an algebra homomorphism, the top horizontal map is
an algebra isomorphism and the bottom horizontal map is a set
bijection. Thus the right vertical map is automatically an algebra
inclusion and the bottom horizontal map is automatically an algebra
isomorphism. The general case follows from the case $A=\bbZ$ by
extension of scalars.
\end{proof}
\begin{rk}
The results of this section are also true for each quiver
$\Gamma$ without loops.
\end{rk}
\subsection{Functor $Y'$\label{subs_Y'}}
We consider two different gradings on the algebra $R_{\nu,A}$:
\begin{enumerate}
    \item the grading introduced in Section \ref{subs_KLR},
    \item the grading
$ \bigoplus_{n\in\bbN}\Ext_{G_V}^n(\calL_{V,A},\calL_{V,A}), $ see
Theorem \ref{thm_isom-KLR-Ext}.
\end{enumerate}
Denote by $\mathrm{mod}(R_{\nu,A})$ (resp.
$\mathrm{mod}'(R_{\nu,A})$) the category of finitely generated
$\mathbb{Z}$-graded $R_{\nu,A}$-modules with respect to the first
(resp. second) grading. Denote by $\mathrm{proj}(R_{\nu,A})$ (resp.
$\mathrm{proj}'(R_{\nu,A})$) the category of projective finitely
generated $\mathbb{Z}$-graded $R_{\nu,A}$-modules with respect to
the first (resp. second) grading. Consider the following
contravariant functor:
$$
Y':D_{G_V}(E_V,A)\to {\rm mod}'(R_{\nu,A}),~\mathfrak{L}\mapsto {\rm
Ext}_{G_V}^*(\mathfrak{L},\calL_{V,A}).
$$
\begin{thm}
\label{thm_cat-equiv-nongrad} The restriction of the functor $Y'$ to
$\mathcal{Q}_V$ yields an equivalence of categories
$\mathcal{Q}_V^{\mathrm{op}}\to \mathrm{proj}'(R_{\nu,A})$.
\end{thm}
\begin{proof}
Suppose that $\calF$ is a direct factor of the complex $\calL_{V,A}$
in $D_{G_V}(E_V,A)$. We have $\calL_{V,A}=\calF\oplus\calG$ for some
complex $\calG$ in $D_{G_V}(E_V,A)$. We have
$$
R_{\nu,A}={\rm Ext}_{G_V}^*(\calL_{V,A},\calL_{V,A})={\rm
Ext}_{G_V}^*(\calF,\calL_{V,A})\oplus {\rm
Ext}_{G_V}^*(\calG,\calL_{V,A})=Y'(\calF)\oplus Y'(\calG).
$$
Thus $Y'(\calF)$ is a projective and finitely generated
$R_{\nu,A}$-module. From this we see that for every $X\in \calQ_V$
the $R_{\nu,A}$-module $Y'(X)$ is projective and finitely generated.
The $R_{\nu,A}$-module $Y'(X)$ is graded by
$Y'(X)_r=\Ext^r_{G_V}(X,\calL_{V,A})$ for each $r\in\bbZ$. Now we
prove that the functor $Y'\colon \calQ^{\mathrm{op}}_V\to
\mathrm{proj}'R_{\nu,A}$ is an isomorphism on morphisms. We need to
show that for $R,S\in\mathcal{Q}_V$ the morphism
$$
Y'_{R,S}:\mathrm{Hom}_{\mathcal{Q}_V}(R,S)\to
\mathrm{Hom}_{\mathrm{proj}'(R_{\nu,A})}(\mathrm{Ext}_{G_V}^*(S,\calL_{V,A}),\mathrm{Ext}_{G_V}^*(R,\calL_{V,A}))
$$
is an isomorphism. First, we consider the case $R=S=\calL_{V,A}$.
Then we have
$$
\mathrm{Hom}_{\mathcal{Q}_V}(R,S)=\mathrm{Hom}_{\mathcal{Q}_V}(\calL_{V,A},\calL_{V,A}),
$$
\begin{eqnarray*}
\mathrm{Hom}_{\mathrm{proj}'(R_{\nu,A})}(\mathrm{Ext}_{G_V}^*(S,\calL_{V,A}),\mathrm{Ext}_{G_V}^*(R,\calL_{V,A})) & = & \mathrm{Hom}_{\mathrm{proj}'(R_{\nu,A})}(R_{\nu,A},R_{\nu,A})\\
& = & R^0_{\nu,A}\\
& = & \mathrm{Ext}^0_{G_V}(\calL_{V,A},\calL_{V,A})\\
& = & \mathrm{Hom}_{\mathcal{Q}_V}(\calL_{V,A},\calL_{V,A}).
\end{eqnarray*}
Here $R^0_{\nu,A}$ means the $0$-graded component of $R_{\nu,A}$
with respect to the second grading. So for $R=S=\calL_{V,A}$ the
statement is true. Then it is also true if $R$ and $S$ are direct
factors of $\calL_{V,A}$. Then it is also true if $R$ and $S$ are
shifts of direct factors of $\calL_{V,A}$. Finally we see that the
statement is true for arbitrary $R,S\in \mathcal{Q}_V$.

To conclude we need to show that $Y'\colon
\mathcal{Q}_V^{\mathrm{op}}\to \mathrm{proj}'(R_{\nu,A})$ is
surjective on isomorphism classes. Let $P$ be a module in
$\mathrm{proj}'(R_{\nu,A})$. A choice of a finite number of
homogeneous generators of $P$ yields an epimorphism
$$
\pi:\bigoplus\limits_{i\in J}R_{\nu,A}[n_i]\to P
$$
for some finite set $J$ and some integers $n_i$. By projectivity
this epimorphism splits. Consider the complex
$\mathfrak{L}=\bigoplus\limits_{i\in I}\calL_{V,A}[-n_i]$. We have
$$
Y'(\mathfrak{L})=\mathrm{Hom}(\mathfrak{L},\calL_{V,A})=\bigoplus\limits_{i\in
J}R_{\nu,A}[n_i].
$$
The composition of the natural projection and inclusion
$$
Y'(\mathfrak{L})\stackrel{\pi_e}{\to} P\stackrel{i_e}{\to}
Y'(\mathfrak{L})
$$
yields an idempotent $e\in
\mathrm{End}_{\mathrm{proj}'(R_{\nu,A})}(Y'(\mathfrak{L}))$. The
functor $Y'\colon \mathcal{Q}_V^{\mathrm{op}}\to
\mathrm{proj}'(R_{\nu,A})$ is an isomorphism on morphisms. So we
have an idempotent $u\in\mathrm{End}_{\mathcal{Q}_V}(\mathfrak{L})$
such that $Y'(u)=e$. The idempotents $u$ and $1-u$ split in
$D_{G_V}(E_V,A)$ yielding a decomposition $\mathfrak{L}=X\oplus Y$,
where $X,Y\in D_{G_V}(E_V,A)$ and $u$ is a composition of the
natural projection and the natural inclusion $
\mathfrak{L}\stackrel{\pi_u}\to X\stackrel{i_u}{\to} \mathfrak{L}. $
By definition of $\mathcal{Q}_V$ the objects $X,Y$ belong to
$\mathcal{Q}_V$. Now we verify that the map
$\pi_e\circ Y'(i_u):~Y'(X)\to P$
is an isomorphism. Let us prove that the map $Y'(\pi_u) \circ i_e:~P\to Y'(X)$
is its inverse. We have
\begin{eqnarray*}
\pi_e\circ Y'(i_u)\circ Y'(\pi_u) \circ i_e & = & \pi_e\circ Y'(u) \circ i_e\\
& = & \pi_e\circ e \circ i_e\\
& = & (\pi_e\circ i_e)\circ (\pi_e \circ i_e)\\
& = & \mathrm{Id}_P
\end{eqnarray*}
and
\begin{eqnarray*}
Y'(\pi_u) \circ i_e\circ \pi_e\circ Y'(i_u) & = & Y'(\pi_u) \circ e \circ Y'(i_u)\\
& = & Y'(\pi_u) \circ Y'(u) \circ Y'(i_u)\\
& = & Y'(\pi_u\circ u\circ i_u)\\
& = & Y'((\pi_u\circ i_u)\circ (\pi_u\circ i_u))\\
& = & Y'(\mathrm{Id}_X)\\
& = & \mathrm{Id}_{Y'(X)}.
\end{eqnarray*}
\end{proof}

\subsection{Induction and restriction\label{subs_res-ind-KLR}}
Let $\bfk$ be a field.

In this section we describe the induction and the restriction
functors for representations of KLR-algebras. See \cite[Sec.~2.6]{KL} for more details. Consider $\nu_1,\nu_2\in\bbN I$ and set
$\nu=\nu_1+\nu_2$. We set
$$
|\nu_1|=m_1,\qquad |\nu_2|=m_2,\qquad |\nu|=m.
$$
There is a unique inclusion of graded $\bfk$-algebras
$R_{\nu_1,\bfk}\otimes R_{\nu_2,\bfk}\subset R_{\nu,\bfk}$
such that
$$
1_{\ui}\otimes 1_{\uj}\mapsto 1_{\uk},\qquad x_{\ui}(k)\otimes
1_{\uj}\mapsto x_{\uk}(k), \qquad 1_{\ui}\otimes x_{\uj}(k)\mapsto
x_{\uk}(m_1+k),
$$
$$
\tau_{\ui}(l)\otimes 1_{\uj}\mapsto \tau_{\uk}(l),\qquad
1_{\ui}\otimes \tau_{\uj}(l)\mapsto \tau_{\uk}(m_1+l)
$$
for each $k,l,\ui,\uj$. Let $1_{\nu_1,\nu_2}$ be the image of the
identity element by this inclusion. We get the following functors:
$$
{\Ind}_{\nu_1,\nu_2}:
\mathrm{mod}(R_{\nu_1,\bfk})\times\mathrm{mod}(R_{\nu_2,\bfk})\to\mathrm{mod}(R_{\nu,\bfk}),\quad
(M_1,M_2)\mapsto
R_{\nu,\bfk}1_{\nu_1,\nu_2}\bigotimes_{R_{\nu_1,\bfk}\otimes
R_{\nu_2,\bfk}}(M_1\otimes M_2),
$$
$$
{\Res}_{\nu_1,\nu_2}:\mathrm{mod}(R_{\nu,\bfk})\to\mathrm{mod}(R_{\nu_1,\bfk}\otimes
R_{\nu_2,\bfk}), \qquad M\mapsto 1_{\nu_1,\nu_2}M.
$$
These functors take projective modules to projective modules.
\subsection{Projective $R_{\nu,\mathbf{k}}$-modules\label{subs_proj-KLR-mod}} Denote by $\mathcal{A}$ the ring $\mathbb{Z}[q,q^{-1}]$. For $m\in \mathbb{N},~m>0$ consider the following elements of $\mathcal{A}$
$$
[m]=\sum_{l=1}^mq^{m+1-2l}=\frac{q^m-q^{-m}}{q-q^{-1}},\qquad
[m]!=\prod_{l=1}^m[l], \qquad l_{m}=m(m-1)/2.
$$
We also set $[0]!=1$, $l_0=0$. For $\mathbf{m}=(m_1,\cdots,m_k)$,
$m_l\in \bbN$ set
$$
[{\mathbf m}]!= \prod_{l=1}^k[m_l]!,\qquad
l_{\mathbf{m}}=\sum_{l=1}^kl_{m_l}.
$$
Given a pair $y =(\ui,\mathbf{a})\in Y_\nu$ we define a projective
$R_{\nu,\mathbf{k}}$-module $R_y$ as follows
\begin{itemize}
    \item If $I=\{i\}$, $\nu=mi$, $\ui=(i,i,\cdots,i)$ and $y=(i,m)$ we set
$$
\bfP_y=\bfP_{i,m}=\calP ol_{\nu,A}[l_m].
$$
$$
R_{\nu,\bfk}\simeq\oplus_{w\in \frakS_m}\bfP_{i,m}[2l(w)-l_m].
$$
So $\bfP_{i,m}$ is a direct summand of $R_{\nu,\bfk}[l_m]$. We
choose once and for all an idempotent $1_{i,m}\in R_{\nu,\bfk}$ such
that
$$
\bfP_{i,m}=(R_{\nu,\bfk}\cdot 1_{i,m})[l_m].
$$
    \item If $y=(\ui,\mathbf{a})$ with $\ui=(i_1,\cdots,i_k)$, $\mathbf{a}=(a_1,\cdots,a_k)$ we define the idempotent $1_y\in R_{\nu,\bfk}$
as the image of the element $\otimes_{l=1}^k1_{i_l,a_l}$ by the
inclusion of graded $\bfk$-algebras $\otimes_{l=1}^k
R_{i_la_l,\bfk}\subset R_{\nu,\bfk}$. Then we set
$$
\bfP_y=(R_{\nu,\bfk}\cdot 1_y)[l_{\mathbf{a}}].
$$
\end{itemize}
We have the following lemma, see \cite[Sec.~2.5-2.6]{KL}.
\begin{lem}
\label{lem_proj-mod-isom} We have the following graded projective
$R_{\nu,\bfk}$-module isomorphisms
\begin{itemize}
    \item[(a)] $\bfP_{\ui}\simeq[\mathbf a]!\bfP_y$, where $y=(i_1^{(a_1)}\cdots i_k^{(a_k)})\in Y_\nu$, $\ui=(i_1^{a_1}\cdots i_k^{a_k})\in I^\nu$, $\mathbf a=(a_1,\cdots,a_k)$,
    \item[(b)] $\Ind_{\nu_1,\nu_2}(\bfP_{y_1},\bfP_{y_2})\simeq\bfP_{y_1y_2}$, where $\nu_1,\nu_2\in \bbN I$, $y_1\in Y_{\nu_1},y_2\in Y_{\nu_2}$ and $y_1y_2\in Y_{\nu_1+\nu_2}$ is the concatenation of $y_1$ and $y_2$.
\end{itemize}
\end{lem}
\subsection{The algebra $\bff$\label{subs_alg-f}}
We recall the definition and general properties of the negative part
$\bff$ of the Drinfeld-Jimbo quantized enveloping algebra associated
with the quiver $\Gamma$. See \cite{Lu} for more details.
\begin{df}
The algebra $\bff$ is the $\bbQ(q)$-algebra generated by the
elements $\theta_i$, $i\in I$ with the relations
$$
\sum_{a+b=1-i\cdot
j}(-1)^a\theta_i^{(a)}\theta_j\theta_i^{(b)},\qquad i\ne
j,~a\geqslant 0,b\geqslant 0,
$$
where $\theta_i^{(a)}=\theta_i^a/[a]!$.
\end{df}

The assignment $\deg\theta_i=i$ for $i\in I$ defines a
$\bbN[I]$-grading $ \bff=\bigoplus_{\nu\in\bbN[I]}\bff_\nu. $ For
each homogeneous element $x\in \bff_\nu$ we set $|x|=\sum_{i\in I}
\nu_i$. The tensor product (over $\bbQ(q)$) ${\bff}\otimes{\bff}$
has a $\bbQ(q)$-algebra structure defined by
$$
(x_1\otimes x_2)(x_1'\otimes x_2')=q^{-|x_2||x_1'|}x_1x_1'\otimes
x_2x_2'
$$
where $x_1,x_2,x_1',x_2'\in {\bff}$ are homogeneous. There exists a
unique coproduct
$r\colon \bff\to\bff\otimes\bff$
such that
\begin{itemize}
    \item $r(\theta_i)=\theta_i\otimes 1+1\otimes \theta_i$ for $i\in I$,
    \item $r$ is a $\bbQ(q)$-algebra homomorphism.
\end{itemize}
\begin{df}
Let $_\mathcal{A}\mathbf{f}$ be the $\mathcal{A}$-subalgebra of
$\mathbf{f}$ generated by the elements $\theta_i^{(a)}$ with $i\in
I,a\in \mathbb{N}$.
\end{df}

\subsection{A new $\bbZ$-basis in $_\calA\bff$\label{subs_new-bas}}
The construction of this section is very similar to the construction given in Section \ref{subs_Y'}. Let
$K(\calQ_V)$ be the split Grothendieck group of the additive
category $\calQ_V$, i.e., $K(\calQ_V)$ is the Abelian group with one
generator $[\calF]$ for each isomorphism class of objects of
$\calQ_V$ and with relations $[\calL']+[\calL'']=[\calL]$ whenever
$\calL$ is isomorphic to $\calL'\oplus\calL''$. Set
$$
\calQ=\bigoplus_{V}\calQ_V,\qquad K(\calQ)=\bigoplus_{V}K(\calQ_V),
$$
where $V$ runs over the isomorphism classes of finite dimensional
$I$-graded $\bfk$-vector spaces and
$$
R_\bfk=\bigoplus_{\nu\in \bbN[I]}R_{\nu,\bfk},\qquad
K(R_\bfk)=\bigoplus_{\nu\in \bbN[I]}K(R_{\nu,\bfk}).
$$
The functors ${\Ind}_{\nu_1,\nu_2}$ induce an associative unital
$\mathcal{A}$-algebra structure on $K(R_\bfk)$. The functors
${\Res}_{\nu_1,\nu_2}$ induce a coassociative counital
$\mathcal{A}$-coalgebra structure on $K(R_\bfk)$. We will denote the
multiplication and the comultiplication on $K(R_\bfk)$ by $\Ind$ and
$\Res$ respectively.
\begin{rk}
There exists a non-degenerate $\calA$-bilinear form
$(,)\colon K(R_\bfk)\times K(R_\bfk)\to \bbQ(q)$
such that
$$
(x,\Ind(y_1,y_2))=(\Res(x),y_1\otimes y_2), \qquad \forall
x,y_1,y_2\in K(R_\bfk),
$$
see \cite[Sec.~2.5]{KL}. 
\end{rk}

The following theorem is proved in \cite[Prop.~3.4,~Sec.~3.2]{KL}.
\begin{thm}
\label{thm_isom-f-KLR} There is an $\mathcal{A}$-bialgebra
isomorphism $\gamma_{\calA}:{_\mathcal{A}\bff}\to K(R_\bfk)$ that
takes $\theta_y$ to $[\bfP_y]$ for each $y\in Y_\nu$.
\end{thm}

The grading of $R_{\nu, \bfk}$ from Section \ref{subs_KLR} is
different from the $\Ext$-grading, see Section \ref{subs_Y'}. To
avoid this we must modify a bit the complex $\calL_{V,\bfk}$. We get
the following theorem, see \cite[Thm.~3.6]{VV}. Recall the
complex $^\delta\calL_{V,\bfk}$ from Section \ref{subs_QV}.
\begin{thm}
\label{thm_isom-grad-KLR-sh} We have the following graded
$\bfk$-algebra isomorphism
$$
R_{\nu,\bfk}\to
\Ext_{G_V}^*({^\delta\calL_{V,\bfk}},{^\delta\calL_{V,\bfk}}).
$$

\end{thm}

The proof of the following theorem is the same as the proof of
Theorem \ref{thm_cat-equiv-nongrad}.
\begin{thm}
\label{thm_cat-equiv-grad} We have the following equivalence of
categories.
$$
Y\colon \calQ_V^{\mathrm{op}}\to
\mathrm{proj}(R_{\nu,\mathbf{k}}),\qquad \frakL\mapsto
\Ext_{G_V}^*(\frakL,{^\delta \calL_{V,\bfk}}).
$$
\end{thm}
So we have an $\calA$-linear map
$Y\colon K(\calQ)\to K(R_\bfk)$ such that $[\calF]\mapsto [Y(\calF)]$.
Moreover, it is clear from the definition of $Y$ that $Y(^\delta
\calL_{\ui})=P_{\ui}$ for $\ui\in I^\nu$. If $\ui\in I^\nu$ is the
expansion of the element $y=(\uj,a)$ in $Y_\nu$ than we have
${^\delta \calL_{\ui}}=[a]!{^\delta \calL_y}$ (see \cite[(4.7)]{VV})
and $[P_{\ui}]=[a]![P_y]$ (see Lemma \ref{lem_proj-mod-isom}). Thus,
$Y(^\delta \calL_y)=P_{y}$ for each $y\in Y_\nu$. So we have
$\gamma_{\calA}^{-1}\circ Y([^\delta \calL_y])=\theta_y$ for $y\in
Y_\nu$, where $\gamma_{\calA}$ is as in Theorem
\ref{thm_isom-f-KLR}.

Consider $\nu_1,\nu_2\in \bbN I$. Let $V_1,~V_2$ be $I$-graded
$\bbC$-vector spaces with graded dimensions $\nu_1,~\nu_2$
respectively. Analogically with the case when the characteristic of
$\bfk$ is zero, see \cite[Sec.~9.2.7]{Lu},  we have a
multiplication
$$
\circ\colon\calQ_{V_1}\times \calQ_{V_1}\to \calQ_{V_1\oplus V_2},
$$
such that ${^\delta \calL_{y_1}}\circ {^\delta \calL_{y_2}}={^\delta
\calL_{y_1y_2}}$ for $y_1\in Y_{\nu_1}$, $y_2\in Y_{\nu_2}$. Thus,
$\gamma_{\calA}^{-1}\circ Y$ is an algebra homomorphism. We get the
following theorem.
\begin{thm}
\label{thm_isom-sh-f} There exists an algebra isomorphism $
\lambda_{\calA}\colon K(\calQ)\to {_\calA\bff} $ such that for each
$y\in Y_\nu$, $\nu\in \bbN I$ we have
$\lambda_\calA([^\delta\calL_y])=\theta_y$.
\end{thm}

For each $I$-graded finite dimensional $\bbC$-vector space $V$ the
indecomposable complexes in $\calQ_V$ form a $\bbZ$-basis in
$K(\calQ_V)$. Combining this for all $V$ we get a $\bbZ$-basis in
$K(\calQ)$. The homomorphism $\lambda_\calA$ takes this $\bbZ$-basis
to a $\bbZ$-basis of $_\calA\bff$.

\subsection{Indecomposable objects in $\calQ_V$\label{subs_indec-QV}}

For each $\lambda\in \Lambda_V$ there is $y_\lambda\in Y_\nu$ such
that $\widetilde{F}_{y_\lambda}\to \overline{\calO_\lambda}$ is a
resolution of the orbit closure $\overline{\calO_\lambda}$, see
\cite[Thm.~2.2]{Rein}. Let $i_\lambda$ and $\tilde{i}_\lambda$ be
the inclusions
$i_\lambda\colon \calO_\lambda\to 
E_V$ and $\tilde{i}_\lambda\colon(\pi_{y_\lambda})^{-1}(\calO_\lambda)\to
\widetilde{F}_{y_\lambda}$.
Denote also $\pi'_{y_\lambda}$ the morphism
$\pi'_{y_\lambda}\colon(\pi_{y_\lambda*})^{-1}(\calO_\lambda)\to
\calO_\lambda$
induced by $\pi_{y_\lambda}$. By the base change theorem we have
$$
i_\lambda^*\pi_{y_\lambda*}\underline{\bfk}_{\widetilde{F}_y}=\pi'_{y_\lambda*}\tilde{i}_\lambda^*\underline{\bfk}_{\widetilde{F}_y}=\underline{\bfk}_\lambda.
$$
So the complex
$\calL_{y_\lambda}=\pi_{y_\lambda*}\underline{\bfk}_{\widetilde{F}_y}$
has a unique indecomposable factor supported on all
$\overline{\calO_\lambda}$. Denote by $R_{y_\lambda}$ the shift by
$d_\lambda$ of this direct factor. We have
$i_\lambda^*R_{y_\lambda}=\underline{\bfk}_\lambda[d_\lambda]$. Note
that $R_{y_\lambda}$ is a shift of a direct factor of $\calL_\ui$
for some $\ui\in I^\nu$, because if $\ui=(i_1^{a_1}\cdots
i_k^{a_k})$ is the expansion of $y=(i_1^{(a_1)}\cdots i_k^{(a_k)})$
then we have ${^\delta\calL}_\ui=[a]!{^\delta\calL}_y$, where
$\mathbf a=(a_1\cdots,a_k)$, see Section \ref{subs_new-bas}. Thus,
$R_{y_\lambda}$ belongs to $\calQ_V$. For $\lambda\in\Lambda_V$ we
denote by $\mathrm{IC}(O_\lambda)$ the $G_V$-equivariant
intersection cohomology complex on $E_V$ associated with
$\calO_\lambda$ and the $G_V$-equivariant local system
$\underline{\bfk}_\lambda$, see \cite[Sec.~8.4]{CG}.

\begin{lem}
\label{lem_L_lamb-indep} The complex $R_{y_\lambda}$ does not depend
on the choice of the element $y_\lambda\in Y_\nu$ such that
$\pi_{y_\lambda}$ is a resolution of the orbit closure
$\overline{\calO_\lambda}$.
\end{lem}
\begin{proof}
Fix an element $y_\lambda$ as above for each $\lambda\in\Lambda_V$.
Suppose that $\bfk$ is a field of characteristic zero. Then by the
decomposition theorem \cite[Thm.~6.2.5]{BBD} we have
$R_{y_\lambda}=\mathrm{IC}(\calO_\lambda)$. Thus, the complex
$\mathrm{IC}(\calO_\lambda)$ belongs to $\calQ_V$. Hence, in view of
Corollary \ref{coro_loc-sys-triv} the complexes
$R_{y_\lambda}=\mathrm{IC}(\calO_\lambda)$ with $\lambda\in
\Lambda_V$, are the representatives of the isomorphism classes of
indecomposable objects in $\calQ_V$ modulo shifts, and they form an
$\calA$-basis in $K(\calQ_V)$. This yields the equality
$$
\dim_{\calA}(_\calA\bff_\nu)=\#(\Lambda_V). \eqno (2.2)
$$
Now let $\bfk$ be an arbitrary field. The number of indecomposable
objects modulo shifts in $\calQ_V$ is equal to
$\dim_{\calA}K(\calQ_V)=\dim_{\calA}(_\calA\bff_\nu)$, see Theorem
\ref{thm_isom-sh-f}, and the complexes $R_{y_\lambda},~
\lambda\in\Lambda_V$ are among them. Thus, $(2.2)$ shows that
$R_{y_\lambda},~ \lambda\in\Lambda_V$ are the representatives of
isomorphism classes of indecomposable objects in $\calQ_V$ modulo
shifts. Hence, each indecomposable object in $\calQ_V$ with support
equal to $\overline{\calO_\lambda}$ is equal to a shift of
$R_{y_\lambda}$. From this we see that $R_{y_\lambda}$ does not
depend on $y_\lambda$.
\end{proof}
\begin{df}
 Let us write $R_\lambda$ instead of $R_{y_\lambda}$.
\end{df}
The following lemma follows directly from the proof of Lemma
\ref{lem_L_lamb-indep}.
\begin{lem}
\label{lem_indec-QV} The indecomposable objects in $\calQ_V$ are
exactly the complexes $R_{\lambda}$, $\lambda\in\Lambda_V$ modulo
shifts.
\end{lem}

\section{Parity sheaves}
Let $A$ be a complete discrete valuation ring or a field.
\subsection{Parity sheaves\label{subs_par-sh}}
First, we recall some basic facts about parity sheaves. See
\cite{PSh} for more details. Let $G$, $X$ be as in Section
\ref{subs_strat-var}. Here we will use the notation introduced in
Section \ref{subs_strat-var}. We make the following assumption
$${\rm H}_G^\mathrm{odd}(X_\lambda,\mathcal{L})=0,~{\rm H}_G^*(X_\lambda,\mathcal{L}) \text{ is a free $A$-module}\quad\forall \mathcal{L}\in\mathrm{Loc}_G(X_\lambda), \forall \lambda\in\Lambda.\eqno (3.1)$$

\begin{df}
\label{def_ev-compl} A complex $\mathcal{F}\in D_G(X,A)$ is
\emph{even} (resp. \emph{odd}) if
$\mathcal{H}^r(\mathcal{F})=\mathcal{H}^r(\mathbb{D}\mathcal{F})=0$
for each odd (resp. even) positive integer $r$. A complex
$\mathcal{F}$ is \emph{parity} if it is a sum of an even complex and
an odd complex.
\end{df}

\begin{lem}
\label{lem_def-ev-equiv} Suppose $\mathcal{F}\in D_G(X,A)$. The
following conditions are equivalent.
\begin{itemize}
    \item[(a)] $\mathcal{F}\in D_G(X,A)$ is even,
    \item[(b)] $\calH^k(i_\lambda^*\calF)=\calH^k(i_\lambda^*\bbD\calF)=0$ for each odd integer $k$ and each $\lambda\in\Lambda$,
    \item[(c)] $\calH^k(i_x^*\calF)=\calH^k(i_x^*\bbD\calF)=0$ for each odd integer $k$ and each $x\in X$.
\end{itemize}
\end{lem}
\begin{proof}
First we show that (a)$\Leftrightarrow$(c) By \cite[Rem.~2.6.9]{KSch} we have
$$
i_x^*\mathcal{H}^k(\mathcal{F})=\calH^k(i_x^*\mathcal{F}).
$$
Now $\mathcal{H}^k(\mathcal{F})$ is zero iff
$i_x^*\mathcal{H}^k(\mathcal{F})$ is zero for each $x\in X$. Thus,
$\mathcal{H}^k(\mathcal{F})=0$ iff $\calH^k(i_x^*\mathcal{F})=0$ for
all $x\in X$. In the same way we prove that
$\mathcal{H}^k(\mathbb{D}\mathcal{F})=0$ iff
$\calH^k(i_x^*\bbD\mathcal{F})=0$ for all $x\in X$. Thus,
(a)$\Leftrightarrow$(c). To see that (b)$\Leftrightarrow$(c) we
prove in the same way that $\calH^k(i_\lambda^*\calF)=0$ iff
$\calH^k(i_x^*\mathcal{F})=0$ for all $x\in X_\lambda$ and
$\calH^k(i_\lambda^*\bbD\calF)=0$ iff
$\calH^k(i_x^*\bbD\mathcal{F})=0$ for all $x\in X_\lambda$.
\end{proof}

\begin{lem}
\label{lem_unic-par-sh} Assume that {\rm (3.1)} holds. Let
$\mathcal{F}$ be an indecomposable parity complex. Then
\begin{enumerate}
    \item the support of $\mathcal{F}$ is of the form $\overline{X_\lambda}$ for some $\lambda\in\Lambda$,
    \item the restriction $i_\lambda^*\mathcal{F}$ is isomorphic to $\mathcal{L}[m]$ for some indecomposable object $\mathcal{L}$ in ${\rm Loc}_G(X_\lambda)$ and some integer m,
    \item any indecomposable parity complex supported on $\overline{X_\lambda}$ which extends $\mathcal{L}[m]$ is isomorphic to $\mathcal{F}$, where $\mathcal{L}$ is as in $(2)$.
\end{enumerate}
\end{lem}

\begin{proof}[Idea of proof]
The lemma is proved in \cite[Thm.~2.12]{PSh}. This proof uses the
following claim, which is based on the property (3.1).
\begin{lem}
Assume that {\rm (3.1)} holds. Let $\calF,\calG\in D_G(X,A)$.
Suppose that $\calH^{\mathrm{odd}}(\calF)=0$ and
$\calH^{\mathrm{odd}}(\bbD\calF)=0$. Then we have
$$
\Hom(\calF,\calG)\simeq\bigoplus_{\lambda\in\Lambda}\Hom(i_\lambda^*\calF,i_\lambda^!\calG).
$$
\end{lem}
\end{proof}

\begin{df}
\label{dem_par-sh}
 A \emph{parity sheaf} is an indecomposable parity complex supported on $\overline{X_\lambda}$ extending $\mathcal{L}[d_\lambda]$ for some indecomposable $\mathcal{L}\in{\rm Loc}_{G}(X_\lambda,A)$ and some $\lambda\in\Lambda$. If such a complex exists, we will denote it by $\mathcal{E}(\lambda,\mathcal{L})$. If $\mathcal{L}$ is the constant sheaf $\underline{A}_\lambda$ we will write $\mathcal{E}(\lambda)=\mathcal{E}(\lambda,\calL)$.
\end{df}

\subsection{Parity sheaves on quiver varieties\label{subs_par-sh-on-quiv}}
We still suppose that $\Gamma$ is a Dynkin quiver. We want to study
$G_V$-equivariant parity sheaves on $E_V$ with respect to the
stratification
$\coprod\limits_{\lambda\in\Lambda_V}\mathcal{O}_\lambda$, see
Section \ref{subs_QV}. We must check that the $G_V$-variety $X=E_V$
satisfy the condition $(3.1)$. This is true by the following lemma,
see Corollary \ref{coro_loc-sys-triv}.
\begin{lem}
\label{lem_prop*} For every $G_V$-orbit $\mathcal{O}$ in $E_V$ we
have
$\mathrm{H}_{G_V}^\mathrm{odd}(\mathcal{O},\underline{A}_\lambda)=0$
and $\mathrm{H}_{G_V}^*(\mathcal{O},\underline{A}_\lambda)$ is a
free $A$-module.
\end{lem}
\begin{proof}
Let $x\in \mathcal{O}$. Denote by $H$ the stabilizer of $x$ in
$G_V$. Let $H_\mathrm{red}$ be the reductive part of $H$. We have
\begin{eqnarray*}
\mathrm{H}_{G_V}^*(\mathcal{O},\underline{A}_{\calO}) & = & \mathrm{H}_{G_V}^*(G_V/H,A)\\
& = & \mathrm{H}_H^*(\bullet,A)\\
& = & \mathrm{H}_{H_{\mathrm{red}}}^*(\bullet,A)\\
& = & \mathrm{H}^*(BH_{\mathrm{red}},A).
\end{eqnarray*}

Note that for each positive integer $n$ the fundamental group of
$GL_n(\mathbb{C})$ is equal to $\mathbb{Z}$. Then the fundamental
group of $H_{\mathrm{red}}$ does not contain elements of finite
order, see Lemma \ref{lem_stab-GL}. Moreover, the reductive group
$H_{\mathrm{red}}$ is of type $A$. Then by \cite[Lem.~1]{Totaro}
the torsion index of $H_{\mathrm{red}}$ is equal to $1$, see Section
\ref{subs_alg-S-P}. Thus, the classifying space $BH_{\mathrm{red}}$
has no odd cohomology and its cohomology is free over $A$, see
\cite[Cor.~2.3]{Tor}.
\end{proof}
The following lemma is helpful to prove the parity of some complexes
on $E_V$.

\begin{lem}
\label{lem_ev-resol} Suppose that there exists a smooth
$G_V$-variety $Y$ and a $G_V$-equivariant proper morphism $
\pi\colon Y\to E_V. $ Then the complex $\pi_*\underline{A}_Y$ is
even if and only if $ \mathrm{H}^{\mathrm{odd}}(\pi^{-1}(x),A)=0 $
for each $x\in E_V$.
\end{lem}
\begin{proof}
Note that the complex $\pi_{*}\underline{A}_{Y}$ is constructible
with respect to the stratification of $E_V$ because it is
$G_V$-equivariant and the strata are $G_V$-orbits. By Lemma
\ref{lem_def-ev-equiv} the complex $\pi_*\underline{A}_Y$ is even if
and only if for every inclusion $j:\{x\}\hookrightarrow E_V$ we have
$ \mathrm{H}^\mathrm{odd}(j^*\pi_{*}\underline{A}_{Y})=0 $ and $
\mathrm{H}^\mathrm{odd}(j^*\mathbb{D}\pi_{*}\underline{A}_{Y})=0. $
Let $d$ be the complex dimension of $Y$. Note that
$$
\mathbb{D}\pi_{*}\underline{A}_{Y}=\pi_{*}\mathbb{D}\underline{A}_{Y},\eqno(3.2)
$$
$$
\mathbb{D}\underline{A}_{Y}=\underline{A}_{Y}[2d],\eqno(3.3)
$$
because the morphism $\pi$ is proper. We set $F=\pi^{-1}(x)$. We
have the following commutative diagram
$$
\begin{CD}
Y @>\pi>> E_V\\
@AiAA                    @AjAA\\
F               @>\pi>> \{x\},
\end{CD}
$$
where $i$ is the inclusion of $F$ to $Y$. We denote the restriction
of $\pi$ to $F$ again by $\pi$. By the base change theorem we have
$$
j^*\pi_{*}\underline{A}_{Y}=\pi_{*}i^*\underline{A}_{Y},\eqno(3.4)
$$
$$
j^*\pi_{*}\mathbb{D}\underline{A}_{Y}=\pi_{*}i^*\mathbb{D}\underline{A}_{Y}.\eqno(3.5)
$$

Now, from $(3.2)$, $(3.3)$, $(3.4)$ and $(3.5)$ we get
$$
\mathrm{H}^*(\{x\},j^*\pi_{*}\underline{A}_{Y})=\mathrm{H}^*(F,A),
$$
$$
\mathrm{H}^*(\{x\},j^*\mathbb{D}\pi_{*}\underline{A}_{Y})=\mathrm{H}^{*+2d}(F,A).
$$
This completes the proof.
\end{proof}

\subsection{Extensions of parity complexes\label{subs_ext-par-sh}}
Let $\calF$ be a parity complex in $D_{G_V}(E_V,A)$. For each $\lambda\in \Lambda_V$ and $n\in\bbN$ the sheaf $\calH^n(i_\lambda^*\calF)$ is a direct sum of copies of $\underline{A}_\lambda$, see Corollary \ref{coro_loc-sys-triv}. Let us denote the rank of the sheaf $\calH^n(i_\lambda^*\calF)$ by $d_{\lambda,n}(\calF)$. 

\begin{lem}
\label{lem_d-par-compl} Let $\calF,\calG$ be parity complexes in
$D_{G_V}(E_V,A)$. Suppose that we have
$d_{\lambda,n}(\calF)=d_{\lambda,n}(\calG)$ for all
$\lambda\in\Lambda_V,n\in\bbN$.
Then $\calF\simeq\calG$.
\end{lem}
\begin{proof}
Let us decompose $\calF$ and $\calG$ into direct sums of shifts of
parity sheaves
$$
\calF=\bigoplus_{\lambda\in\Lambda_V,k\in\bbZ}\calE(\lambda)[k]^{\oplus
f_{\lambda,k}},\qquad
\calG=\bigoplus_{\lambda\in\Lambda_V,k\in\bbZ}\calE(\lambda)[k]^{\oplus
g_{\lambda,k}}.
$$
We will prove the statement by induction on the number of
indecomposable factors of $\calF$. If $\calF=0$ then for each
$\lambda\in\Lambda_V,n\in\bbZ$ we have
$d_{\lambda,n}(\calF)=d_{\lambda,n}(\calG)=0$. Thus,
$\calF=\calG=0$. Now suppose that $\calF\ne 0$. Define a partial
order $\preccurlyeq$ on $\Lambda_V$ by setting $\mu\preccurlyeq
\lambda$ if and only if $\calO_\mu\subset\overline{\calO_\lambda}$.
Fix a total order $\leqslant$ on $\Lambda_V$ that refines
$\preccurlyeq$. Let $\lambda_0$ be the largest element of
$\Lambda_V$ with respect to $\leqslant$ such that
$f_{\lambda_0,k}\ne 0$ for some integer $k$. Then the decomposition
of $\calF$ does not contain any shift of the parity sheaves
$\calE(\mu)$ with $\lambda_0<\mu$. If for some $\mu,\lambda\in
\Lambda_V$ and $n\in\bbN$ we have $\calH^n(i_\mu^*\calE(\lambda))\ne
0$ then $\mu\preccurlyeq \lambda$. Thus, we have
$d_{\mu,n}(\calF)=0$ for each $\mu>\lambda_0,n\in\bbZ$. Now, the
equality  $d_{\mu,n}(\calG)=0$ for each $\mu>\lambda_0,n\in\bbZ$
assures that the decomposition of $\calG$ does not contain any shift
of the parity sheaves $\calE(\mu)$ with $\lambda_0<\mu$. Choose
$n_0$ such that
$d_{\lambda_0,n_0}(\calF)=d_{\lambda_0,n_0}(\calG)\ne 0$. We have
$$
\calH^{n_0}(i_{\lambda_0}^*\calF)\simeq
\underline{A}_{\lambda_0}^{\oplus f_{\lambda_0,k_0}},\qquad
\calH^{n_0}(i_{\lambda_0}^*\calG)\simeq
\underline{A}_{\lambda_0}^{\oplus g_{\lambda_0,k_0}}
$$
for $k_0=-d_{\lambda_0}-n_0$. This shows that
$$
f_{\lambda_0,k_0}=d_{\lambda_0,n_0}(\calF)=d_{\lambda_0,n_0}(\calG)=g_{\lambda_0,k_0}.
$$
Thus, we have
$$
\calF=\calF'\oplus \calE(\lambda_0)[k_0]^{\oplus
f_{\lambda_0,k_0}},\qquad \calG=\calG'\oplus
\calE(\lambda_0)[k_0]^{\oplus f_{\lambda_0,k_0}}.
$$
Thus, we have $d_{\lambda,n}(\calF')=d_{\lambda,n}(\calG')$ for each
$\lambda\in\Lambda_V,n\in\bbZ$. Thus, by induction hypothesis we
have $\calF'\simeq\calG'$. Hence, $\calF\simeq\calG$.
\end{proof}

\begin{lem}
\label{lem_ext-par-compl} Let $ A\to B\to C\stackrel{+1}{\to} $ be a
distinguished triangle in $D_{G_V}(E_V,A)$. Suppose that the
complexes $A$ and $C$ are even. Then we have $B\simeq A\oplus C$. In
particular the complex $B$ is also even.
\end{lem}

\begin{proof}
See also \cite[Cor.~2.8]{PSh}. First we prove that $B$ is an
even complex. The duality yields a distinguished triangle
$$
\bbD C\to\bbD B\to\bbD A\stackrel{+1}{\to}.
$$
The long exact sequences in cohomology are
$$
\to\calH^{n-1}(C)\to \calH^n(A)\to \calH^n(B)\to\calH^n(C)\to
\calH^{n+1}(A)\to,
$$
$$
\to\calH^{n-1}(\bbD A)\to \calH^n(\bbD C)\to \calH^n(\bbD
B)\to\calH^n(\bbD A)\to \calH^{n+1}(\bbD C)\to.
$$
For each odd integer $n$ we have
$$
\calH^n(A)=\calH^n(\bbD A)=\calH^n(C)=\calH^n(\bbD C)=0
$$
and thus $\calH^n(B)=\calH^n(\bbD B)=0$. So the complex $B$ is also
even. For each $\lambda\in\Lambda_V$ the distinguished triangle
$$
i_\lambda^*A\to i_\lambda^* B\to i_\lambda^*C\stackrel{+1}{\to}
$$
yields a long exact sequence in cohomology
$$
\to\calH^{n-1}(i_\lambda^*C)\to \calH^n(i_\lambda^*A)\to
\calH^n(i_\lambda^*B)\to\calH^n(i_\lambda^*C)\to
\calH^{n+1}(i_\lambda^*A)\to.
$$
If $n$ is an even integer then
$\calH^{n-1}(i_\lambda^*C)=\calH^{n+1}(i_\lambda^*A)=0$, see Lemma
\ref{lem_def-ev-equiv}. Thus, we have a short exact sequence
$$
0\to \calH^n(i_\lambda^*A)\to
\calH^n(i_\lambda^*B)\to\calH^n(i_\lambda^*C)\to 0.
$$
Hence, we have
$$
d_{\lambda,n}(B)=d_{\lambda,n}(A)+d_{\lambda,n}(C)=d_{\lambda,n}(A\oplus
C)
$$
for each $\lambda\in \Lambda_V$ and each even integer $n$ and
$$
d_{\lambda,n}(A)=d_{\lambda,n}(B)=d_{\lambda,n}(C)=0
$$
for each odd integer $n$. Thus, by Lemma \ref{lem_d-par-compl} we
have $B\simeq A\oplus C$.
\end{proof}

\subsection{Nakajima quiver varieties\label{subs_Nak-quiv}} Let $\bfk$ be a field. From now all sheaves and cohomology groups are to be understood with coefficients in $\bfk$-vector spaces.

We recall the notion of a (graded) Nakajima quiver variety. We
keep the notation of Section \ref{subs_quiv} and we still suppose
that the quiver $\Gamma$ is a Dynkin quiver. The notation $i \sim j$
means that there is an arrow $i\to j$ or $j\to i$ in $\Gamma$. Fix a function $\xi\colon I\to \bbZ,~i\mapsto \xi_i $ such that
$\xi_i=\xi_j+1$ if there is an arrows $i\to j$. It is unique modulo
a shift by an integer constant. Consider the sets
$$
\widehat I=\{(i,n)\in I\times \bbZ;~\xi_i-n\in 2\bbZ\},
$$
$$
\widehat J=\{(i,n)\in I\times \bbZ;~(i,n-1)\in \widehat I\}.
$$
To the quiver $\Gamma$ we will associate the quiver
$\widehat{\Gamma}$ as follows:
\begin{itemize}
    \item the quiver $\widehat{\Gamma}$ contains two types of vertices:
\begin{itemize}
    \item a vertex $w_i(n)$ for each $(i,n)\in \widehat I$, we will call them \emph{w-vertices},
    \item a vertex $t_i(n)$ for each $(i,n)\in \widehat J$, we will call them \emph{t-vertices},
\end{itemize}
    \item the quiver $\widehat{\Gamma}$ contains three types of edges:
\begin{itemize}
    \item an edge $w_i(n)\to t_i(n-1)$ for each $(i,n)\in \widehat{I}$,
    \item an edge $t_i(n)\to w_i(n-1)$ for each $(i,n)\in \widehat{J}$,
    \item an edge $t_i(n)\to t_j(n-1)$ for each $(i,n),(j,n-1)\in \widehat{J}$ such that $i\sim j$.
\end{itemize}
\end{itemize}

Consider an $\widehat I$-graded finite dimensional $\bbC$-vector
space $ W=\bigoplus_{(i,n)\in \widehat I}W_i(n) $ and a $\widehat
J$-graded finite dimensional $\bbC$-vector space $
T=\bigoplus_{(i,n)\in \widehat J}T_i(n). $ Set
$$
L^\bullet (T,W)=\bigoplus_{(i,n)\in \widehat J}
\Hom(T_i(n),W_i(n-1)),
$$
$$
L^\bullet (W,T)=\bigoplus_{(i,n)\in \widehat I}
\Hom(W_i(n),T_i(n-1)),
$$
$$
E^\bullet (T)=\bigoplus_{(i,n)\in \widehat J, i\sim j}
\Hom(T_i(n),T_j(n-1)),
$$
and put
$$
M^\bullet (T,W)=E^\bullet (T)\oplus L^\bullet (W,T)\oplus L^\bullet
(T,W).
$$
Note that $M^\bullet (T,W)$ is just the variety of representations
of $\widehat{\Gamma}$ in the $\widehat I\coprod \widehat J$-graded
vector space $T\oplus W$. We will write $(B,\alpha,\beta)$ for an
element of $M^\bullet (T,W)$. The components of $B, \alpha, \beta$
are
$$
B_{ij}(n)\in \Hom(T_i(n),T_j(n-1)),
$$
$$
\alpha_i(n)\in \Hom(W_i(n),T_i(n-1)),
$$
$$
\beta_i(n)\in \Hom(T_i(n),W_i(n-1)).
$$
Let $\Lambda^\bullet(T,W)$ be the subvariety of the affine space
$M^\bullet(T,W)$ defined by the equations
$$
\alpha_i(n-1)\beta_i(n)+\sum_{i\sim
j}\varepsilon(i,j)B_{ji}(n-1)B_{ij}(n)=0, \qquad \forall(i,n)\in
\widehat J,
$$
where $\varepsilon(i,j)=1$ (resp. $\varepsilon(i,j)=-1$) if $i\to j$
is an arrow of $\Gamma$ (resp. $i\to j$ is not an arrow of
$\Gamma$).

The algebraic group $ G_T=\prod_{(i,n)\in \widehat J}GL_{T_i(n)} $ acts
on $M^\bullet (T,W)$ by $ g\cdot
(B,\alpha,\beta)=(B',\alpha',\beta'), $ where
$$
B'_{ij}(n)=g_j(n-1)B_{ij}(n)g_i(n)^{-1},\quad
\alpha'_i(n)=g_i(n-1)\alpha_i(n),\quad
\beta'_i(n)=\beta_i(n)g_i(n)^{-1}.
$$
Consider the categorical quotient
$$
\frakM_0^\bullet(T,W)=\Lambda^\bullet(T,W)\slash\!\!\slash G_T,
$$
i.e., the coordinate ring of $\frakM_0^\bullet(T,W)$ is the ring of
$G_T$-invariant functions on $\Lambda^\bullet(T,W)$. The closed
points of $\frakM_0^\bullet(T,W)$ parametrize the closed
$G_T$-orbits in $\Lambda^\bullet(T,W)$. We will denote by
$[B,\alpha,\beta]$ is the element of $\frakM_0^\bullet(T,W)$
represented by $(B,\alpha,\beta)\in\Lambda^\bullet(T,W)$. We say
that a point $(B,\alpha,\beta)\in \Lambda^\bullet(T,W)$ is stable if
the following condition holds: if a $\widehat J$-graded subspace
$T'$ of $T$ is $B$-invariant and contained in $\Ker \beta$, then
$T'=0$. Let $\Lambda_{\mathrm s}^\bullet(T,W)$ be the set of stable
points in $\Lambda^\bullet(T,W)$. Consider the set theoretical
quotient
$$
\frakM^\bullet(T,W)=\Lambda_{\mathrm s}^\bullet(T,W)/G_T.
$$
This quotient coincides with a quotient in the geometric invariant
theory (see \cite[Sec.~3]{N0}). There exists a projective
morphism
$$
\pi_{T,W}\colon\frakM^\bullet(T,W)\to \frakM_0^\bullet(T,W),
$$
see \cite[Sec.~3.18]{N0}. For each $\widehat J$-graded subspace $T'$
of $T$ we have a closed embedding
$$
\frakM_0^\bullet(T',W)\subset \frakM_0^\bullet(T,W).
$$
This allows to define the variety
$$
\frakM_0^\bullet(\infty,W)=\bigcup\limits_T\frakM_0^\bullet(T,W).
$$

Let $\frakM_0^{\bullet\mathrm{reg}}(T,W)$ be the open subset of
$\frakM_0^\bullet(T,W)$ parametrizing the closed free $G_T$-orbits
in $\Lambda^\bullet(T,W)$. Denote by $\bfS(W)$ the set of
isomorphism classes of finite dimensional $\widehat J$-graded vector
spaces $T$ such that $\frakM_0^{\bullet\mathrm{reg}}(T,W)\ne
\emptyset$. For each $W$ as above we have a decomposition
$$
\frakM_0^{\bullet}(\infty,W)=\coprod_{T\in \bfS(W)}
\frakM_0^{\bullet\mathrm{reg}}(T,W),
$$
see \cite[(4.5)]{N3}.


\subsection{Resolutions via Nakajima quiver varieties\label{subs_res-Nak}}
\begin{df}
Suppose $(i,n),(j,k)$ are in $\widehat{I}$. A \emph{path} from
$(i,n)$ to $(j,k)$ is a path in the quiver $\widehat{\Gamma}$ from
$w_i(n)$ to $w_j(k)$ that does not contain another $w$-vertex. To a
path $p$ from $(i,n)$ to $(j,k)$ and an element $(B,\alpha,\beta)\in
\Lambda^\bullet(T,W)$ we can associate an element
$A_{p}(B,\alpha,\beta)\in\Hom(W_i(n),W_j(k))$ as follows. If
$$
p=(w_i(n)\to t_{i_1}(n-1)\to t_{i_2}(n-2)\to\cdots\to
t_{i_r}(n-r)\to w_{j}(k)),
$$
where $r=n-k-1$, then we set
$$
A_{p}(B,\alpha,\beta)=\beta_{j}(n-r)B_{i_{r-1}i_r}(n-(r-1))\cdots
B_{i_1i_2}(n-1)\alpha_{i}(n)\colon W_i(n)\to W_j(k).
$$
\end{df}
The following theorem is proved in \cite[Prop.~9.4,~Thm.~9.11]{HL}.
\begin{thm}
\label{thm_isom-Nak-Lus-quiv} There exists an injective map
$\tau\colon I\to \widehat I$ (depending on the orientation of
$\Gamma$) such that for each arrow $i\to j$ in $\Gamma$ there exists
a path $i\leadsto j$ from $\tau(i)$ to $\tau(j)$ such that the
following property holds: to each finite dimensional $I$-graded
$\bbC$-vector space $V$ we assign the $\widehat I$-graded
$\bbC$-vector space $W$ by setting
$$
W_{\tau(i)}=V_i, \quad\forall i\in I,\qquad W_{\hat \iota}=0
\quad\forall  \hat \iota\in \widehat I\backslash \tau(I).
$$
We have an isomorphism
$$
\frakM_0^\bullet(\infty,W)\to E_V,\qquad [B,\alpha,\beta]\mapsto
\bigoplus_{i\to j} A_{i\leadsto j}(B,\alpha,\beta).
$$

\end{thm}

From now for each arrow $i\to j$ in $\Gamma$ we fix a path
$i\leadsto j$ from $\tau(i)$ to $\tau(j)$ as in Theorem
\ref{thm_isom-Nak-Lus-quiv}.

\begin{rk}
\label{rem_isom-unic-path} The isomorphism in Theorem
\ref{thm_isom-Nak-Lus-quiv} depends on the choice of the path
$i\leadsto j$ from $\tau(i)$ to $\tau(j)$ for each arrow $i\to j$ in
$\Gamma$. Given another path $p$ from $\tau(i)$ to $\tau(j)$, the
morphisms
$$
A_{i\leadsto j}\colon \frakM_0^\bullet(\infty,W)\to \Hom(V_i,V_j),
\qquad A_{p}\colon \frakM_0^\bullet(\infty,W)\to \Hom(V_i,V_j)
$$
are such that $A_{p}=\lambda A_{i\leadsto j}$ for some
$\lambda\in\bbC$. This follows from two following lemmas.
\end{rk}
For each arrow $i\to j$ in $\Gamma$ we equip the variety
$\Hom(V_i,V_j)$ with the $G_V$-action by $g\cdot \phi=g_j\phi
g_i^{-1}$ for $g\in G_V$, $\phi\in\Hom(V_i,V_j)$.
\begin{lem}
\label{lem_isom-unic-path-lin} The map
$$
A_{p}\colon E_V=\frakM_0^\bullet(\infty,W)\to \Hom(V_i,V_j)
$$
is linear and $G_V$-equivariant.
\end{lem}
\begin{proof}
The statement is obvious.
\end{proof}
\begin{lem}
\label{lem_isom-unic-path-const} Let $h_0=(i\to j)$ be an arrow of
$\Gamma$. Then each linear $G_V$-equivariant morphism $ f\colon
E_V\to \Hom(V_i,V_j) $ is of the form $ x\mapsto \lambda x_{h_0} $
for some $\lambda\in \bbC$.
\end{lem}
\begin{proof}
We view $E_V$ and $\Hom(V_i,V_j)$ as linear representations of
$G_V$. The representation $E_V$ is a direct sum of irreducible
representations $ E_V=\bigoplus_{h\in H}E_h, $ where
$E_h=\Hom(V_{h'},V_{h''})$ and $h'$, $h''$ are as in Section
\ref{subs_quiv}. In particular, we have $E_{h_0}=\Hom(V_i,V_j)$.
Now, by Schur's lemma we have
$$
\Hom_{G_V}(E_h,E_{h_0})=
\begin{cases}
0 & \text{if } h\ne h_0,\\
\bbC \cdot \mathrm{Id}_{E_{h_0}} & \text{if }  h=h_0.
\end{cases}
$$
\end{proof}

The following lemma is proved in \cite[Thm.~7.4.1]{N1}.
\begin{lem}
\label{lem_fib-Nak} Let $T$ and $W$ be as in Section
\ref{subs_Nak-quiv}. Then the fibers of $\pi_{T,W}$ have no odd
cohomology groups over $\bbZ$ and their cohomology groups over $\bbZ$
are torsion free.
\end{lem}
Let us associate $W$ to $V$ as in Theorem
\ref{thm_isom-Nak-Lus-quiv}. We will identify $E_V$ with
$\frakM_0^{\bullet}(\infty,W)$ as in Theorem
\ref{thm_isom-Nak-Lus-quiv}. The following lemma is proved in
\cite[Prop.~9.8]{HL}. It interprets the stratification on
$E_V$ in terms of Nakajima quiver varieties.
\begin{lem}
\label{lem_strat-Nak} The decomposition $
\frakM_0^{\bullet}(\infty,W)=\coprod_{T\in\bfS(W)}
\frakM_0^{\bullet\mathrm{reg}}(T,W), $ coincides with the
decomposition of $E_V$ into $G_V$-orbits.
\end{lem}
Let $T$ be a finite dimensional $\widehat J$-graded $\bbC$-vector space.
The morphism $\pi_{T,W}$ defined in Section \ref{subs_Nak-quiv}
together with the inclusion $\frakM_0^{\bullet}(T,W)\subset
\frakM_0^{\bullet}(\infty,W)$ yields a morphism
$$
\pi_{T,W}\colon \frakM^{\bullet}(T,W)\to
\frakM_0^{\bullet}(\infty,W)\simeq E_V.
$$
We can consider the following complex in $D_{G_V}(E_V,\bfk)$.
$$
\widetilde\scrL_{T,W}=\pi_{T,W*}\underline\bfk_{\frakM^{\bullet}(T,W)}
$$
Lemmas \ref{lem_ev-resol}, \ref{lem_fib-Nak} and \ref{lem_strat-Nak}
yield the following result.
\begin{coro}
\label{coro_Nak-sh-even} The complex $\widetilde\scrL_{T,W}$ is an
even complex in $D_{G_V}(E_V,\bfk)$.
\end{coro}
The following lemma assures the existence of the parity sheaves on
$E_V$, see Definition \ref{dem_par-sh}
\begin{lem}
\label{lem_exist-par-sh} For each $\lambda\in\Lambda_V$ there exists
a parity sheaf $\calE(\lambda)$ on
$E_V=\coprod_{\lambda\in\Lambda_V}\calO_\lambda$.
\end{lem}
\begin{proof}

Let $T$ be in $\bfS(W)$, see Section \ref{subs_Nak-quiv}. By
\cite[Prop.~3.24]{N0} the morphism
$$
\pi_{T,W}\colon\frakM^{\bullet}(T,W)\to \frakM_0^{\bullet}(\infty,W)
$$
is an isomorphism over $\frakM_0^{\bullet\mathrm{reg}}(T,W)\subset
\frakM_0^{\bullet}(T,W)$. Its image is contained in
$\overline{\frakM_0^{\bullet\mathrm{reg}}(T,W)}$ because
$\frakM^{\bullet}(T,W)$ is smooth and connected by \cite[Thm.~5.5.6]{N1} and $\pi_{T,W}^{-1}(\frakM_0^{\bullet\mathrm{reg}}(T,W))$
is a dense open subset of $\frakM^\bullet(T,W)$. Moreover, the complex
$\widetilde\scrL_{T,W}$ is even by Corollary \ref{coro_Nak-sh-even}.
Thus, the base change argument in Section \ref{subs_indec-QV} shows
each parity sheaf $\calE(\lambda), ~\lambda\in \Lambda_V$ is
well-defined and appears as a direct factor of some
$\widetilde\scrL_{T,W}$ for some $T$.
\end{proof}

\begin{rk}
\label{rem_indec-compl} \text{ }
\begin{itemize}
    \item[(a)]
Lemma \ref{lem_exist-par-sh} is enough to insure the existence of
$G_V$-equivariant parity sheaves on $E_V$, see Corollary
\ref{coro_loc-sys-triv}. Note that in the same way we can prove the
existence of $G_V$-equivariant parity sheaves on $E_V$ over a
complete discrete valuation ring.
    \item[(b)]
We have $ i_\lambda^* R_\lambda=i_\lambda^*
\calE(\lambda)=\underline{\bfk}_\lambda[d_\lambda], $ see Section
\ref{subs_indec-QV}. It is natural to ask whether the indecomposable
complexes $\calE(\lambda)$ and $R_\lambda$ are isomorphic. To get
this, it is enough to prove that $R_\lambda$ is a parity complex,
see Lemma \ref{lem_unic-par-sh}.
    \item[(c)] For $\lambda\in \Lambda_V$ set $\widetilde\calE(\lambda)=\calE(\lambda)[-d_\lambda]$. Let $T\in \bfS(W)$ be the $\widehat J$-graded $\bbC$-vector space such that $\frakM^{\bullet\mathrm{reg}}_0(T,W)\simeq\calO_\lambda$. Suppose that the characteristic of $\bfk$ is zero. Then by decomposition theorem \cite[Thm.~6.2.5]{BBD} the unique indecomposable factor of $\scrL_{T,W}$ supported on all $\overline{\calO_\lambda}$ is $\mathrm{IC}(\calO_\lambda)[-d_\lambda]$, because we have
$ i_\lambda^* \widetilde\scrL_{T,W}=i_\lambda^*
\widetilde\calE(\lambda)=\underline{\bfk}_\lambda. $ On the other
hand $\widetilde\scrL_{T,W}$ is the sum of shifts of parity sheaves
and thus the same argument shows that this factor is equal to
$\widetilde\calE(\lambda)$. Thus, we have
$\mathrm{IC}(\calO_\lambda)\simeq\calE(\lambda)$. In particular in
characteristic zero the complex
$\mathrm{IC}(\calO_\lambda)[-d_\lambda]$ is even. Moreover, we have
already seen in the proof of Lemma \ref{lem_L_lamb-indep} that
$R_\lambda\simeq\mathrm{IC}(\calO_\lambda)$. Thus, in zero
characteristic we have $ \calE(\lambda)\simeq
R_\lambda\simeq\mathrm{IC}(\calO_\lambda). $
\end{itemize}
\end{rk}

\subsection{Restriction diagrams\label{subs_restr-diag}}
In this section we recall the construction of the restriction
diagrams for Nakajima quiver varieties and Lusztig quiver
varieties and we compare them.

Let $W$ be a finite dimensional $\widehat I$-graded $\bbC$-vector
space. Let $W_2$ be an $\widehat I$-graded subspace of $W$ and set
$W_1=W/W_2$. Consider the subvariety of $\frakM_0^\bullet(\infty,W)$
given by
$$
\frakZ_0^\bullet(W_1,W_2)=\{[B,\alpha,\beta]\in\frakM_0^\bullet(\infty,W);~
\beta B^k\alpha(W_2)\subset W_2 ~\forall k\in\bbN\},
$$
where $[B,\alpha,\beta]$ denotes the element of
$\frakM_0^\bullet(\infty,W)$ represented by $(B,\alpha,\beta)\in
\Lambda^\bullet(V,W)$ for some $V$. Consider the following diagram,
see \cite[Sec.~3.5]{VV2} and \cite[Sec.~3.5]{N4}.
$$
\frakM_0^\bullet(\infty,W_1)\times\frakM_0^\bullet(\infty,W_2)\stackrel{\kappa'}{\leftarrow}\frakZ_0^\bullet(W_1,W_2)\stackrel{\iota'}{\rightarrow}\frakM_0^\bullet(\infty,W),
$$
where $\iota'$ is the inclusion and $\kappa'$ is the induced map
defined as follows. Fix the element
$[B,\alpha,\beta]\in\frakZ_0^\bullet(W_1,W_2)$ represented by the
tuple $(B,\alpha,\beta)\in\Lambda^\bullet(T,W)$ for some $T$. Let
$T'$ be the maximal ($\widehat J$-graded) subspace of $T$ stable by
$B$ such that $\beta(T')\subset W_2$ to $W_2$. We have
$\alpha(W_2)\in T'$ by definition of $\frakZ_0^\bullet(W_1,W_2)$.
Then $(B,\alpha,\beta)$ defines elements of
$\Lambda^\bullet(T/T',W_1)$ and  $\Lambda^\bullet(T',W_2)$. These
elements yield elements of $\frakM_0^\bullet(\infty,W_1)$ and
$\frakM_0^\bullet(\infty,W_2)$.

Now let $V$ be an $I$-graded finite dimensional $\bbC$-vector space.
Let $V_2$ be an $I$-graded subspace of $V$ and set $V_1=V/V_2$.
Consider the subvariety $F$ of $E_V$ consisting of the
representations preserving $V_2$. We consider the following diagram,
see \cite[Sec.~9.2]{Lu}
$$
E_{V_1}\times
E_{V_2}\stackrel{\kappa}{\leftarrow}F\stackrel{\iota}{\rightarrow}E_V,
$$
where $\iota$ is the inclusion and $\kappa$ assigns to an elements
of $F$ its quotient and restriction. Now assign $W$ to $V$ as in
Theorem \ref{thm_isom-Nak-Lus-quiv}. The subspace $V_2\subset V$
yields a subspace $W_2\subset W$ and the quotient $W_1=W/W_2$ would
correspond to $V_1=V/V_2$.

\begin{thm}
\label{thm_isom-diag} We have an isomorphism of diagrams
$$
\begin{CD}
\frakM_0^\bullet(\infty,W_1)\times\frakM_0^\bullet(\infty,W_2)@<\kappa'<<\frakZ_0^\bullet(W_1,W_2) @>\iota'>>\frakM_0^\bullet(\infty,W),\\
@VVV         @VVV             @VVV\\
E_{V_1}\times E_{V_2} @<\kappa<< F @>\iota>>E_V,
\end{CD}
$$
where the leftmost and the rightmost isomorphisms are as in Theorem
\ref{thm_isom-Nak-Lus-quiv}.
\end{thm}

\begin{proof}
Consider the element $[B,\alpha,\beta]\in\frakM_0^\bullet(\infty,W)$
represented by the tuple $(B,\alpha,\beta)\in\Lambda^\bullet(T,W)$
for some $T$. Let $i\to j$ be an arrow in $\Gamma$. Recall that the
map $V_i \to V_j$ associated with  $[B,\alpha,\beta]$ is
$$
\beta_{i_k}(n_k)B_{i_{k-1}i_k}(n_{k-1})\cdots
B_{i_0i_1}(n_0)\alpha_{i_0}(n_0+1).
$$
So in view of Remark \ref{rem_isom-unic-path} the conditions that
$[B,\alpha,\beta]\in\frakZ_0^\bullet(W_1,W_2)$ means that the
corresponding element of $E_V$ preserves $V_2$. So we have an
isomorphism $\frakZ_0^\bullet(W_1,W_2)\simeq F$ making the right
part of the diagram in the statement to be commutative. The left
part is also commutative because both $\kappa'$ and $\kappa$ were
defined as the product of the quotient and the restriction.
\end{proof}

The restriction diagram for Lusztig quiver varieties yields a
functor
$$
\widetilde{\Res}_{V_1,V_2}=\kappa_!\iota^*\colon
D_{G_V}(E_V,\bfk)\to  D_{G_{V_1}\times G_{V_2}}(E_{V_1}\times
E_{V_2},\bfk).
$$
Consider also its modification
$$
\Res_{V_1,V_2}=\widetilde\Res_{V_1,V_2}[M_{V_1,V_2}],
$$
where
$$
M_{V_1,V_2}=\sum_{h\in H}\dim (V_1)_{h'}\dim (V_2)_{h''}-\sum_{i\in
I}\dim (V_1)_i\dim (V_2)_i.
$$

\subsection{Restriction of Nakajima sheaves\label{subs_restr-Nak}}
Let $T$ be a $\widehat J$-graded finite dimensional $\bbC$-vector
space. Let $\scrL_{T,W}$ be the shift
$\scrL_{T,W}=\widetilde\scrL_{T,W}[\dim_{\bbC}
\frakM^\bullet(T,W)]$, where $\widetilde\scrL_{T,W}$ is as in
Section \ref{subs_res-Nak}. We identify the restriction diagrams for
Nakajima quiver varieties and Lusztig quiver varieties as in
Theorem \ref{thm_isom-diag}. We will write $\iota=\iota',
\kappa=\kappa'$ to simplify the notation. Set
$$
\widetilde{\frakZ}_0^\bullet(T,W_1,W_2)=(\pi_{T,W})^{-1}(\frakZ_0^\bullet(W_1,W_2)).
$$
We have the commutative diagram
$$
\begin{CD}
\widetilde{\frakZ}_0^\bullet(T,W_1,W_2) @>\tilde\iota>>\frakM^\bullet(T,W)\\
@VVpV             @VV{\pi_{T,W}}V\\
\frakZ_0^\bullet(W_1,W_2) @>\iota>>\frakM_0^\bullet(\infty,W),\\
\end{CD}
$$
where $\tilde\iota$ is the inclusion, $p$ is the restriction of
$\pi_{T,W}$.

Set
$$
U(T)=\{(\omega_1,\omega_2)\in \bbN\widehat J\times\bbN\widehat
J;~\omega_1+\omega_2=\grdim T\}.
$$
For each pair $(\omega_1,\omega_2)\in U(T)$ set
$$
\widetilde
F(\omega_1,\omega_2)=\{[B,\alpha,\beta]\in\widetilde{\frakZ}_0^\bullet(T,W_1,W_2);~\grdim
T/T'=\omega_1,\grdim T'=\omega_2\},
$$
where the subset $T'\subset T$ is associated with $[B,\alpha,\beta]$
as in the definition of the restriction diagram for Nakajima
quiver varieties (see the definition of the morphism $\kappa'$). For
each $(\omega_1,\omega_2)\in U(T)$ we have a morphism
$$
\alpha_{\omega_1,\omega_2}\colon \widetilde F(\omega_1,\omega_2)\to
\frakM^\bullet(T_1,W_1)\times\frakM^\bullet(T_2,W_2),
$$
where $T_1$, $T_2$ are $\widehat J$-graded $\bbC$-vector spaces
with dimension vectors $\omega_1,\omega_2$ respectively. The
morphism $\alpha_{\omega_1,\omega_2}$ is a vector bundle by
\cite[Prop.~3.8]{N5}. Let $d_{T_1,T_2}$ be its rank. The
following Lemma is a positive characteristic analogue of \cite[Lem.~4.1]{VV2}.
\begin{lem}
\label{lem_restr-Nak} We have
$$
\widetilde{\Res}_{V_1,V_2}(\widetilde\scrL_{T,W})=\bigoplus_{T_1\oplus
T_2\simeq T}\widetilde\scrL_{T_1,W_1}\otimes
\widetilde\scrL_{T_2,W_2}[-2d_{T_1,T_2}],
$$
where the sum is taken over the isomorphism classes of $\widehat
J$-graded $\bbC$-vector spaces $T_1$, $T_2$ such that $T_1\oplus
T_2\simeq T$.
\end{lem}
\begin{proof}
For each $(\omega_1,\omega_2)\in U(T)$ and each $\widehat J$-graded
$\bbC$-vector spaces $T_1$, $T_2$ of graded dimensions $\omega_1$,
$\omega_2$ we have a commutative diagram
$$
\begin{CD}
\widetilde F(\omega_1,\omega_2) @>>>\widetilde{\frakZ}_0^\bullet(T,W_1,W_2)\\
@V\alpha_{\omega_1,\omega_2}VV    @V{\kappa p}VV\\
\frakM^\bullet(T_1,W_1)\times\frakM^\bullet(T_2,W_2)
@>>>\frakM_0^\bullet(\infty,W_1)\times\frakM_0^\bullet(\infty,W_2).
\end{CD}
$$
The upper horizontal arrow is the obvious inclusion and the lower
one is $\pi_{T_1,W_1}\times \pi_{T_2,W_2}$. The varieties
$\widetilde F(\omega_1,\omega_2)$ with $(\omega_1,\omega_2)\in U(T)$
form a locally closed partition of
$\widetilde{\frakZ}_0^\bullet(T,W_1,W_2)$. Let us enumerate them
$\widetilde{F}_1,\cdots, \widetilde{F}_r$ in a such way that for
$j\in [1,r]$ the subvariety $\widetilde F_{\leqslant
j}=\bigcup_{j'\leqslant j} \widetilde F_{j'}$ is closed in
$\widetilde{\frakZ}_0^\bullet(T,W_1,W_2)$. Let $f_j$ and
$f_{\leqslant j}$ be the restrictions of $\kappa p$
$$
f_j\colon \widetilde F_j\to
\frakM_0^\bullet(\infty,W_1)\times\frakM_0^\bullet(\infty,W_2),
\qquad f_{\leqslant j}\colon \widetilde F_{\leqslant j}\to
\frakM_0^\bullet(\infty,W_1)\times\frakM_0^\bullet(\infty,W_2).
$$
We also set $\widetilde F_{<j}=\widetilde F_{\leqslant j-1}$,
$f_{<j}=f_{\leqslant j-1}$ for $j\in[2,r]$. Let $a_j$ and $b_j$ be
the inclusions
$$
a_j\colon \widetilde F_j\to \widetilde F_{\leqslant j}, \qquad
b_j\colon \widetilde F_{<j} \to \widetilde F_{\leqslant j}.
$$
By \cite[Sec.~1.4.3]{BBD}, we have a distinguished triangle
$$
a_{j!}a_j^*\underline{\bfk}_{\widetilde F_{\leqslant
j}}\to\underline{\bfk}_{\widetilde F_{\leqslant j}}\to
b_{j!}b_j^*\underline{\bfk}_{\widetilde F_{\leqslant
j}}\stackrel{+1}{\to}.
$$
Applying the functor $(f_{\leqslant j})_!$ to this triangle we get
the distinguished triangle
$$
(f_j)_!\underline{\bfk}_{\widetilde F_{j}}\to (f_{\leqslant
j})_!\underline{\bfk}_{\widetilde F_{\leqslant j}}\to
(f_{<j})_!\underline{\bfk}_{\widetilde F_{<j}}\stackrel{+1}{\to}.
\eqno (3.6)
$$
Now if $\widetilde F_j=\widetilde F(\omega_1,\omega_2)$ and $T_1$,
$T_2$ are $\widehat J$-graded $\bbC$-vector spaces with graded
dimensions $\omega_1$, $\omega_2$ respectively,  we have
$$
\alpha_{\omega_1,\omega_2!}\underline{\bfk}_{\widetilde
F_j}=\underline{\bfk}_{\frakM^\bullet(T_1,W_1)\times\frakM^\bullet(T_2,W_2)}[-2d_{T_1,T_2}]
$$
because $\alpha_{\omega_1,\omega_2}$ is a vector bundle of rank
$d_{T_1,T_2}$. Thus,
$$
f_{j!}\underline{\bfk}_{\widetilde
F_j}=\widetilde\scrL_{T_1,W_1}\otimes\widetilde\scrL_{T_2,W_2}[-2d_{T_1,T_2}].
$$
In particular it is an even complex, see Corollary
\ref{coro_Nak-sh-even}. Thus, from (3.6) and Lemma
\ref{lem_ext-par-compl} we have by induction
$$
\kappa_!p_!\underline{\bfk}_{\widetilde{\frakZ}_0^\bullet(T,W_1,W_2)}=\bigoplus_{T_1\oplus
T_2\simeq
T}\widetilde\scrL_{T_1,W_1}\otimes\widetilde\scrL_{T_2,W_2}[-2d_{T_1,T_2}].
$$
This finishes the proof because by the base change theorem we have
$$
\kappa_!\iota^*
\widetilde\scrL_{T,W}=\kappa_!p_!\underline{\bfk}_{\widetilde{\frakZ}_0^\bullet(T,W_1,W_2)}.
$$
\end{proof}

The split Grothendieck group $K(\calQ_V)$ of the additive category
$\calQ_V$ is given the $\calA$-module structure such that
$q[\calF]=[\calF[1]]$ for $\calF\in \calQ_V$. We set
$K(\calQ)=\bigoplus_V K(\calQ_V)$. Denote by
$\mathrm{Par}_{G_V}(E_V)$ the full additive subcategory of
$D_{G_V}(E_V,\bfk)$ whose objects are parity complexes. We have an
$\calA$-module structure on the split Grothendieck group
$K(\mathrm{Par}_{G_V}(E_V))$ of the category
$\mathrm{Par}_{G_V}(E_V)$ as above. We also set
$$
\mathrm{Par}=\bigoplus_V \mathrm{Par}_{G_V}(E_V),\qquad
K(\mathrm{Par})=\bigoplus_V K(\mathrm{Par}_{G_V}(E_V)).
$$
Let $\bfS(W)$ be as in Section \ref{subs_Nak-quiv}.
\begin{lem}
\label{lem_bas-parsh-Nak} The complexes $\widetilde \scrL_{T,W}$,
$T\in\bfS(W)$ form an $\calA$-basis in $K(\mathrm{Par}_{G_V}(E_V))$.
\end{lem}
\begin{proof}
Note that
$i_\lambda^*\widetilde\calE(\lambda)=\underline{\bfk}_\lambda$,
where $\widetilde\calE(\lambda)$ is as in Remark
\ref{rem_indec-compl}. Let $\leqslant$ be the total order on
$\Lambda_V$ as in the proof of Lemma \ref{lem_d-par-compl}. Denote
by $T_\lambda$ the $\widehat J$-graded finite dimensional
$\bbC$-vector space such that
$\frakM_0^{\bullet\mathrm{reg}}(T_\lambda,W)\simeq \calO_\lambda$,
see Lemma \ref{lem_strat-Nak}. We have
$\{T_\lambda;~\lambda\in\Lambda_V\}=\bfS(W)$.  The complex
$\widetilde\scrL_{T_\lambda,W}$ is even, see Corollary
\ref{coro_Nak-sh-even}. By the base change theorem we have
$i_\lambda^*\widetilde\scrL_{T_\lambda,W}=\underline\bfk_\lambda$.
Thus,
$$
\widetilde\scrL_{T_\lambda,W}=\widetilde{\calE}(\lambda)\oplus\bigoplus_{\mu<
\lambda, d\in \bbZ}\widetilde{\calE}(\mu)[2d]^{\oplus
a_{\mu,d}},\qquad a_{\mu,d}\in\bbN.
$$
Hence, the classes $[\widetilde\scrL_{T_\lambda,W}]$,
$\lambda\in\Lambda_V$  form an $\calA$-basis in
$K(\mathrm{Par}_{G_V}(E_V))$ because the classes
$[\widetilde{\calE}(\lambda)]$, $\lambda\in\Lambda_V$ form an
$\calA$-basis in $K(\mathrm{Par}_{G_V}(E_V))$.
\end{proof}

By Lemmas \ref{lem_restr-Nak}, \ref{lem_bas-parsh-Nak} the functor
$\widetilde\Res_{V_1,V_2}$ yields an $\calA$-module homomorphism
$$
K(\mathrm{Par}_{G_V}(E_V))\to K(\mathrm{Par}_{G_{V_1}\times
G_{V_2}}(E_{V_1}\times E_{V_2})),
$$
where $\mathrm{Par_{G_{V_1}\times G_{V_2}}(E_{V_1}\times E_{V_2})}$
is defined as $\mathrm{Par}_{G_V}(E_V)$ using the quiver
$\Gamma\coprod\Gamma$. A $G_{V_1}\times G_{V_2}$-orbit in
$E_{V_1}\times E_{V_2}$ is of the form $\calO_{\lambda_1}\times
\calO_{\lambda_2}$. So we identify $\Lambda_{V}$ with
$\Lambda_{V_1}\times \Lambda_{V_2}$. The complex
$\widetilde\calE(\lambda_1)\times \widetilde\calE(\lambda_2)$ is
indecomposable because its endomorphism ring is a tensor product of
endomorphism rings of complexes $\widetilde\calE(\lambda_1)$ and
$\widetilde\calE(\lambda_2)$, thus it is local. It is even,
supported on $\overline{\calO_{\lambda_1}\times \calO_{\lambda_2}}$,
and we have
$$
i_{\lambda_1\times \lambda_2}^*(\widetilde\calE(\lambda_1)\otimes
\widetilde\calE(\lambda_2))=\underline\bfk_{\lambda_1\times
\lambda_2}.
$$
Thus, by Lemma \ref{lem_unic-par-sh}  we have
$$
\widetilde\calE(\lambda_1)\otimes \widetilde\calE(\lambda_2)\simeq
\widetilde\calE(\lambda_1\times\lambda_2).
$$
Thus, we have
$$
K(\mathrm{Par}_{G_{V_1}\times G_{V_2}}(E_{V_1}\times E_{V_2}))\to
K(\mathrm{Par}_{G_{V_1}}(E_{V_1}))\otimes_{\calA}
K(\mathrm{Par}_{G_{V_2}}(E_{V_2}))
$$

Finally, the functors $\widetilde\Res_{V_1,V_2}$ yield a
comultiplication on
$$
K(\mathrm{Par})=\bigoplus_{V} K(\mathrm{Par}_{G_V}(E_V)),
$$
where the sum is taken by isomorphism classes of $I$-graded finite
dimensional $\bbC$-vector spaces. Thus, the modified restriction
functors $\Res_{V_1,V_2}$ yield a coproduct $\Res$ on
$K(\mathrm{Par})$. From now we will always consider the coalgebra
structure on $K(\mathrm{Par})$ given by $\Res$.

\subsection{Restriction of Lusztig sheaves\label{subs_restr-Lus}}
Now we study the restriction  of Lusztig sheaves. It is similar to
the constructions of Nakajima sheaves in Section
\ref{subs_restr-Nak}.
Let
$$
E_{V_1}\times
E_{V_2}\stackrel{\kappa}{\leftarrow}F\stackrel{\iota}{\to}E_V
$$
be the restriction diagram as in Section \ref{subs_restr-diag}. Fix
a sequence $(a_1,\cdots,a_r)$. Set $y=(i_1^{(a_1)}\cdots
i_r^{(a_k)})$, an element in $Y_{\nu}$. Set also
$$
U(y)=\{(y_1,y_2)\in Y_{\nu_1}\times
Y_{\nu_2};~y_1=(i_1^{(a'_1)}\cdots
i_r^{(a'_k)}),y_2=(i_1^{(a''_1)}\cdots i_r^{(a''_k)}),
$$
$$
a'_r,a''_r\in \bbN,~a'_r+a''_r=a_r ~\forall r\in[1,k]\}.
$$
We have a commutative diagram
$$
\begin{CD}
\widetilde F@>>> \widetilde{F}_y\\
@VVV             @VV\pi_yV\\
F           @>>> E_V,
\end{CD}
$$
where $\widetilde F=\pi_y^{-1}(F)$. The variety $\widetilde{F}$
consists of the pairs $(x,\phi)\in E_V\times F_y$ such that $x$
preserves $V_2$ and $x$ preserves the flag $\phi=(\{0\}=V^0\subset
V^1\subset V^2\subset \cdots\subset V^k=V)$. For $(y_1,y_2)$, denote
by $\widetilde{F}(y_1,y_2)$ the subset of $\widetilde F$ consisting
of pairs $(x,\phi)\in \widetilde F$ such that the flag
$(V^0/(V^0\cap V_2)\subset V^1/(V^1\cap V_2)\subset\cdots\subset
V^r/(V^r\cap V_2))$ has type $y_1$ and the flag $(V^0\cap V_2\subset
V^1\cap V_2\subset\cdots\subset V^r\cap V_2)$ has type $y_2$. We
have a vector bundle
$$
\widetilde F(y_1,y_2)\to\widetilde F_{y_1}\times \widetilde F_{y_2},
$$
see \cite[Sec.~9.2.4]{Lu}. Denote its rank by $m_{y_1,y_2}$.

\begin{lem}
\label{lem_restr-Lus} Suppose that for each $(y_1,y_2)\in U(y)$ the
complexes $\calL_{y_1}$, $\calL_{y_2}$ are even. Then we have
$$
\widetilde{\Res}_{V_1,V_2}(\calL_y)=\bigoplus_{(y_1,y_2)\in
U(\nu)}\calL_{y_1}\otimes \calL_{y_2}[-2m_{y_1,y_2}].
$$
\end{lem}
\begin{proof}
The proof is the same as the proof of Lemma \ref{lem_restr-Nak}. See
also \cite[Sec.~9.2]{Lu}.
\end{proof}

\begin{lem}
\label{lem_bas-QV-Lus} For each $\lambda\in \Lambda_V$ let
$y_\lambda\in Y_\nu$ be as in Section \ref{subs_indec-QV}. The
classes $[\calL_{y_\lambda}]$, $\lambda\in \Lambda_V$ form an
$\calA$-basis in $K(\calQ_V)$.
\end{lem}
\begin{proof}
The proof is the same as the proof of Lemma \ref{lem_bas-parsh-Nak}.
\end{proof}


\subsection{Coalgebra structure on $K(\mathrm{Par})$\label{subs_coal-Par}}

The following Lemma follows directly from Lemma
\ref{lem_d-par-compl}.
\begin{lem}
\label{lem_incl-Par} Let $\calA[\Lambda_V\times \bbZ]$ be a free
$\calA$-module with basis $1_{\lambda,n}$, $\lambda\in \Lambda_V$,
$n\in \bbZ$. Then there is an $\calA$-module inclusion
$$
K(\mathrm{Par}_{G_V}(E_V))\hookrightarrow \calA[\Lambda_V\times
\bbZ],\qquad \calF\mapsto
\sum_{\lambda,n}d_{\lambda,n}(\calF)1_{\lambda,n},
$$
where $d_{\lambda,n}$ are as in Section \ref{subs_ext-par-sh}.
\end{lem}

\begin{thm}
\label{thm_coal-isom}
There exists an $\calA$-coalgebra isomorphism $\beta_\calA\colon{_\calA}\bff\to K(\mathrm{Par})$. 
\end{thm}
\begin{proof}
At first suppose $k=\bbC$. In this case the statement is already
known, see \cite[Thm.~13.2.11]{Lu}. Now let $\bfk$ be an
arbitrary field. To avoid confusion we will indicate the field
$\bfk$ or $\bbC$ in the rest of the proof.
By Lemma \ref{lem_fib-Nak} the numbers
$d_{\lambda,n}(\widetilde\scrL_{T,W,\bfk})$ do not depend on the
field $\bfk$, where $\lambda\in \Lambda_V$, $n\in \bbN$, $T$ is a
finite dimensional $\widehat J$-graded $\bbC$-vector space. Thus, the
image of the inclusion in Lemma \ref{lem_incl-Par} does not depend
on the field either. Hence, there exists a linear isomorphism
$$
K(\mathrm{Par}_{G_V}(E_V),\bfk)\to K(\mathrm{Par}_{G_V}(E_V),\bbC)
$$
sending $\widetilde\scrL_{T,W,\bfk}$ to $\widetilde\scrL_{T,W,\bbC}$
for each $T$. To conclude we need only to note that this isomorphism
preserves the coproduct because the constants $d_{T_1,T_2}$ from
Lemma \ref{lem_restr-Nak} are independent of the field.
\end{proof}
\begin{rk}
There is an $\calA$-basis in $K(\mathrm{Par})$ given by parity
sheaves. It corresponds to some $\calA$-basis of $_\calA\bff$ by
$\beta_\calA$. So Theorem \ref{thm_coal-isom} yields an
$\calA$-basis in $_\calA\bff$ in terms of parity sheaves.
\end{rk}

\subsection{Even quivers\label{subs_ev-quiv}}
\begin{df}
The quiver $\Gamma$ is $\bfk$-\emph{even} if the complex $\calL_y$
over $\bfk$ is even for each $y\in Y_\nu$ and each $\nu\in \bbN I$.
The quiver $\Gamma$ is \emph{even} if it is $\bfk$-even for each
field $\bfk$.
\end{df}
In this section we suppose that our Dynkin quiver $\Gamma$ is
even. This hypothesis is necessary to apply Lemma
\ref{lem_restr-Lus}.

\begin{thm}
\label{thm_ev-isom-Nak-Lus} Suppose that $\Gamma$ is an even Dynkin
quiver. Let $V$ be as above. Then the full subcategories $\calQ_V$
and $\mathrm{Par}_{G_V}(E_V)$ of $D_{G_V}(E_V,\bfk)$ coincide.
Moreover, there exists an isomorphism of bialgebras $K(\calQ)\simeq
{_\calA\bff}$, where the algebra structure on $K(\calQ)$ is as in
Section \ref{subs_new-bas} and the coproduct is given by $\Res$.
\end{thm}
\begin{proof}
The first statement follows from Remark \ref{rem_indec-compl} (b).
The second statement is already known if the characteristic of
$\bfk$ is zero, see \cite[Thm.~13.2.11]{Lu}. In positive characteristic
we already have the algebra isomorphism
$$
\lambda_\calA\colon_\calA\bff\to K(\calQ)
$$
by Theorem \ref{thm_isom-sh-f}. We need only to verify that
$\lambda_\calA$ preserves the coproduct. This follows from the zero
characteristic case and Lemma \ref{lem_restr-Lus} because the
constants $m_{y_1,y_2}$ is Lemma \ref{lem_restr-Lus} do not depend on
the field.
\end{proof}

\subsection{Type $A$\label{subs_typeA}} We say that a quiver is of type $A$ if each of its connected components has type $A_n$ for some positive integer $n$. Now we show that the quivers of type $A$ are even. 
We start by proving some helpful lemmas.

\begin{df}
\label{def_varX} Suppose $n,k,s\in \bbN$, $k\leqslant n$, $s>0$.
Suppose $\bfv$ and $\bfd$ are sequences of nonnegative integers
$\bfv=(v_1,\cdots,v_s)$, $\bfd=(d_1,\cdots,d_s)$ such that
$v_1+\cdots+v_s=n$. Let $W$ be an $n$-dimensional $\bbC$-vector
space and let
$$
\phi=(\{0\}=W_0\subset W_1\subset\cdots\subset W_s=W)
$$
be a flag of type $\bfv$ in $W$, i.e., $\dim W_a/W_{a-1}=v_a$ for
each $a$ in $[1,s]$. Set
$$
X_{\bfv,\bfd,k}=\{U\in \Gr_k(W);~\dim U\cap W_a=d_a, \forall
a\in[1,s]\}.
$$
\end{df}

\begin{lem}
\label{lem_varX} The set $X_{\bfv,\bfd,k}$ is either empty or a
variety with an affine cell decomposition.
\end{lem}
\begin{proof}
Suppose that $X_{\bfv,\bfd,k}$ is not empty. Let $P$ be the
parabolic subgroup in $\mathrm{GL}(W)$ preserving the flag $\phi$.
The variety $X_{\bfv,\bfd,k}$ is isomorphic to a $P$-orbit in
$\Gr_k(W)$. Let $B$ be a Borel subgroup contained in $P$. A
$P$-orbit in $\Gr_k(W)$ is a disjoint union of $B$-orbits. Each of
these $B$-orbits is isomorphic to an affine space.
\end{proof}
\begin{df}
\label{def_varY} Let $W$ be a finite dimensional $\bbC$-vector
space. Suppose $s,t\in \bbN$, $s>0,t>0$. Let
$$
\phi=(\{0\}=W_0\subset W_1\subset \cdots\subset W_s=W),\quad
\psi=(\{0\}=V_0\subset V_1\subset \cdots \subset V_t=W)
$$
be two partial flags in $W$. Let $\bfv$ be a sequence of nonnegative
integers $\bfv=(v_1,\cdots,v_s)$ and let $D=(d_{ij})$ be an $s\times
t$-matrix of integers. Set
$$
Y_{\phi,\psi, \bfv,D}=\{(\{0\}=U_0\subset U_1\subset\cdots\subset
U_s=W):
$$
$$
\dim U_a/U_{a-1}=v_a, \dim U_a\cap V_b=d_{a,b}, W_a\subset U_a,
\quad \forall a\in[1,s], b\in[1,t]\}.
$$
\end{df}
\begin{lem}
\label{lem_varY} The set $Y_{\phi,\psi, \bfv,D}$ is either empty or
a variety with an affine cell decomposition.
\end{lem}
\begin{proof}
The proof is by induction on $s$. The case $s=1$ follows from Lemma
\ref{lem_varX}. Suppose $s>1$. Set
$$
\phi'=(\{0\}=W_0\subset W_2\subset W_3\subset \cdots\subset
W_s=W),\quad \bfv'=(v_1+v_2,v_3,\cdots, v_s).
$$
Let $D'$ be the matrix that we get from $D$ by erasing the first
row. Forgetting the $U_1$-component of the flag yields the morphism
$ Y_{\phi,\psi, \bfv,D}\to Y_{\phi',\psi, \bfv',D'}. $ This is a
fibration with the fibre of the form $X_\bullet$, see Definition
\ref{def_varX}. So $Y_{\phi,\psi, \bfv,D}$ has a decomposition into
affine cells by induction hypothesis and Lemma \ref{lem_varX}.
\end{proof}

Suppose that the quiver $\Gamma$ is of type $A_n$. We enumerate its
vertices $I=\{i_1,\cdots, i_n\}$ such that there is an arrow in some
direction between $i_a$ and $i_{a+1}$ for $a\in [1,n-1]$. Let $V$
and $\nu$ be as in Section \ref{subs_quiv}. Suppose
$y=(j_1^{(a_1)},\cdots,j_m^{(a_m)})\in Y_\nu$. Let
$$
\{0\}=W_0\subset W_1\subset\cdots\subset W_k=V_{i_n}
$$
be a partial flag in $V_{i_n}$. For each $m\times k$-matrix
$D=(d_{ij})$ of integers we set
$$
F_{D}=\{(\{0\}=V^0\subset V^1\subset\cdots\subset V^m=V)\in
\pi_y^{-1}(x):
$$
$$
 \dim V^r\cap W_s=d_{rs}, \forall r\in [1,m],s\in [1,k]\}.
$$
We have the decomposition $ \pi_y^{-1}(x)=\coprod_D F_D. $

\begin{lem}
\label{lem_cell-dec-A} The set $F_D$ is either empty or a variety
with an affine cell decomposition.
\end{lem}
\begin{proof}
The proof is by induction on $n$. For $n=1$ the statement follows
from Lemma \ref{lem_varY}. Now suppose $n\geqslant 2$. Suppose that
the quiver $\Gamma$ contains the arrow $i_{n-1}\to i_n$. We denote
this arrow by $h_0$. Consider the following flag in $V_{i_{n-1}}$
$$
\{0\}=W'_0\subset W'_1\subset\cdots\subset W'_{k+1}=V_{i_{n-1}},
\qquad W'_r=x_{h_0}^{-1}(W_{r-1}),~ r\in [1,k],
$$
where $x_{h_0}$ is the ${h_0}$-component of $x$. Let $\Gamma'$ be
the quiver that we get from $\Gamma$ by deleting the vertex $i_n$.
Set $ V'=\bigoplus_{r=1}^{n-1} V_{i_r}, ~\nu'=\nu-\nu_{i_n}\cdot
i_n. $ As before we denote by $H$ the set of arrows in $\Gamma$ and
for an arrow $h$ in $H$ we write $h'$ and $h''$ for its source and
target respectively. Set $H'=H\backslash h_0$. Denote by $E_{V'}$
the variety of representations of $\Gamma'$ in $V'$, i.e., $
E_{V'}=\bigoplus_{h\in H'}\Hom(V_{h'},V_{h''}). $ Set $
x'=\bigoplus_{h\in H'}x_{h}\in E_{V'}, $ where $x_{h}$ is the
${h}$-component of $x\in E_V$. Set $
y'=(j_1^{(a'_1)},\cdots,j_m^{(a'_m)})\in Y_{\nu'}, $ where
$$
a'_t=
\begin{cases}
a_t &\text{if } j_t\ne i_n,\\
0 &\text{if } j_t=i_n,
\end{cases}
\quad\forall t\in[1,m].
$$

We leave the zero terms in $y'$ to simplify the notation. For each
$m\times (k+1)$-matrix of integers $D'=(d_{rs}')$ we set
$$
F'_{D'}=\{(\{0\}={V'}^0\subset {V'}^1\subset\cdots\subset
{V'}^m=V')\in \pi_{y'}^{-1}(x'):
$$
$$
\dim {V'}^r\cap W'_s=d'_{rs}, \forall r\in [1,m],s\in [1,k+1]\}.
$$
We have the morphism $ e\colon \pi_{y}^{-1}(x)\to\pi_{y'}^{-1}(x') $
given by the intersection of each component of the flag with $V'$.
Set
$$
F_{D,D'}=F_D\cap e^{-1}(F'_{D'}).
$$
We have the decomposition $ F_D=\coprod_{D'}F_{D,D'}. $ So it is
enough to construct a decomposition into affine cells for each
$F_{D,D'}$. Let us study the following restriction of $e$
$$
e_{D,D'}\colon F_{D,D'}\to F'_{D'}.
$$
We want to show that it is a fibration. Let $c_1<\cdots < c_t$ be
the integers in $[1,m]$ such that
$$
j_{c_1}=j_{c_2}=\cdots=j_{c_t}=i_n,\qquad j_s\ne i_n~\forall
s\in[1,m]\backslash\{c_1,\cdots, c_t\}.
$$
Set also $c_0=0$. Let
$$
\phi'=(\{0\}=V'^0\subset V'^1\subset \cdots\subset V'^m=V')
$$
be a flag in $F'_{D'}$. We have
\begin{eqnarray*}
e_{D,D'}^{-1}(\phi') &=& \{(\{0\}=V^0\subset V^1\subset \cdots \subset V^{m}=V)\in F_D: \\
&& V^s\cap V'=V'^s, \forall s\in[1,m]\}\\
& = & \{(\{0\}=U^0\subset U^1\subset \cdots \subset U^{m}=V_{i_n}):\\
&&(\{0\}={V'}^0\oplus U^0\subset {V'}^1\oplus U^1\subset \cdots \subset {V'}^{m}\oplus U^m=V)\in F_D\}\\
& = & \{(\{0\}=U^0\subset U^1\subset \cdots \subset U^{m}=V_{i_n}):\\
&& U^s=U^{s-1}, \forall s\in[1,m]\backslash \{c_1,\cdots, c_t\},\\
&&\dim U^s/U^{s-1}=a_s, \forall s\in\{c_1,\cdots, c_t\},\\
&& x_{h_0}(V'^{r}\cap V_{i_{n-1}})\subset U^{r}, \forall r\in[1,m],\\
&&\dim U^r\cap W_p=d_{r,p}, \forall r\in[1,m], p\in[1,k]\}. \\
& = & Y_{\phi,\psi, \bfv, D},
\end{eqnarray*}
where
$$
\phi=(x_{h_0}({V'}^0\cap V_{i_{n-1}})\subset x_{h_0}({V'}^1\cap
V_{i_{n-1}})\subset\cdots\subset x_{h_0}({V'}^{m-1}\cap
V_{i_{n-1}})\subset V_{i_n}),
$$
$$
\psi=(\{0\}=W_0\subset W_1\subset\cdots\subset W_k=V_{i_n}),\qquad
\bfv=(a_1-a'_1,\cdots, a_m-a'_m),
$$
see Definition \ref{def_varY}. Moreover, this fibre does not depend
on the choice of $\phi'\in F'_{D'}$ because the types of the flags
$\phi,\psi$ and their relative position do not depend on $\phi'$.
Really,
\begin{eqnarray*}
\dim x_{h_0}({V'}^a\cap V_{i_{n-1}})\cap W_b & = & \dim x_{h_0}({V'}^a\cap V_{i_{n-1}}\cap x_{h_0}^{-1}(W_b))\\
 & = & \dim x_{h_0}({V'}^a\cap W'_{b+1})\\
 & = & \dim {V'}^a\cap W'_{b+1} -\dim {V'}^a\cap W'_{b+1}\cap\Ker x_{h_0}\\
 & = & \dim {V'}^a\cap W'_{b+1}-\dim {V'}^a\cap W'_{1}\\
 & = & d'_{a,b+1}-d'_{a,1},\qquad \forall a\in[1,m-1], b\in[1,k],\\
\dim x_{h_0}({V'}^a\cap V_{i_{n-1}})& = &\dim x_{h_0}({V'}^a\cap V_{i_{n-1}})\cap W_k\\
& = & d'_{a,k+1}-d'_{a,1},\qquad \forall a\in[1,m-1].\\
\end{eqnarray*}
So $F_{D,D'}$ has a decomposition into affine cells by induction
hypothesis and Lemma \ref{lem_varY}.

Now suppose that the quiver $\Gamma$ contains the arrow $i_n\to
i_{n-1}$. Let $\Gamma^{\mathrm{op}}$ be the quiver that we get from
$\Gamma$ by inverting the orientation of each arrow. Consider the
$I$-graded $\bbC$-vector space $V^*=\bigoplus_{i\in I}V^*_i$ dual to
$V$. Consider the representation $x^*\in E_{V^*}$ dual to $x$. Let
$y^{\mathrm{op}}$ be the element $Y_\nu$ that we get from $y$ by
reversing the order of its components. Then we have the variety
isomorphism
$$
\pi_y^{-1}(x)\to \pi_{y^{\mathrm{op}}}^{-1}(x^*),
$$
$$
(\{0\}=V^0\subset V^1\subset\cdots\subset V^m=V)\mapsto
(\{0\}=(V^m)^\perp\subset \cdots\subset (V^1)^\perp\subset
(V^0)^\perp=V^*).
$$
Here $\bullet^\perp$ denotes the annihilator in the dual space
$V^*$. Moreover, this isomorphism preserves the decompositions
$$
\pi_y^{-1}(x)=\coprod_D F_D,
\quad\pi_{y^{\mathrm{op}}}^{-1}(x^*)=\coprod_D F^*_D
$$
(with some permutation of indices). Here $F^*_D$ is defined
analogically to $F_D$ with respect to the flag
$$
\{0\}=(W_k)^\perp\subset\cdots\subset (W_1)^\perp\subset (W_0)^\perp
=V^*_{i_n},
$$
where $\bullet^\perp$ denotes the annihilator in $V_{i_n}^*$ (not in
$V^*$). So we reduce the case  \newline $i_n\to i_{n-1}$ to the case
$i_{n-1}\to i_{n}$.
\end{proof}
\begin{thm}
\label{thm_cell-dec-A} Let $\Gamma$ be a quiver of type $A$ (maybe
not connected). Suppose $y\in Y_\nu$, $x\in E_V$. Then the fibre
$\pi_y^{-1}(x)$ is either empty or has a decomposition into affine
cells.
\end{thm}
\begin{proof}
It is enough to prove the statement in the case when the quiver
$\Gamma$ is connected. In this case the statement follows from Lemma
\ref{lem_cell-dec-A}.
\end{proof}

\begin{coro}
A quiver of type $A$ is even.
\end{coro}

\subsection{Quiver Schur algebras} In this section, we apply a similar approach to the quiver Schur algebras in \cite{SW}.
Let $e$ be an integer, $e>1$. Let $\Gamma$ be a quiver with the set of vertices $I=\bbZ/e\bbZ$ and the set of arrows $H=\{i\to i+1;~i\in I\}$. We keep the notation of Section \ref{subs_quiv}. Let $E^n_V\subset E_V$ be the subvariety of nilpotent representations.

Let us change slightly the notation.
Let $\mathrm{VComp_e(\nu)}$ be the set of tuples $\mu=(\mu^{(1)},\cdots, \mu^{(k)})$ of nonzero elements of $\bbN I$ such that $\mu^{(1)}+\mu^{(2)}+\cdots+\mu^{(k)}=\nu$. We will call an element of $\mathrm{VComp_e(\nu)}$ a \emph{vector composition}. For $\mu$ as above we denote by $F_\mu$ the variety of all $I$-graded flags
$$
\phi=(0=V^0\subset V^1\subset\cdots\subset V^k=V)
$$
in the $I$-graded vector space $V=\bigoplus_{i\in I}V_i$ such that the $I$-graded vector space $V^k/V^{k-1}$ has graded dimension $\mu^{(r)}$ for each $r\in\{1,2,\cdots,k\}$. Let $\widetilde{F}_\mu$ be the variety of pairs $(x,\phi)\in E_V\times F_\mu$ that are compatible, i.e., we have $x(V^r)\subset V^{r-1}$ for $r\in[1,k]$. Let $\pi_\mu$ be the natural projection from $\widetilde{F}_\mu$ to $E_V$, i.e.,
$$
\pi_\mu:\widetilde{F}_\mu\to E_V,~(x,\phi)\mapsto x.
$$

Let $F$ be the set of all flags in $V$ (not necessarily complete). Let $F_V\subset F$ be the subset of $I$-graded flags. Set
$$
\widetilde{F}=\{(\phi=(V^0\subset\cdots\subset V^k=V),x)\in F\times N; ~x(V^r)\subset V^{r-1}~\forall r\in[1,k]\},
$$
$$
\widetilde{F}_V=\{(\phi=(V^0\subset\cdots\subset V^k=V),x)\in F_V\times E_V^n; ~x(V^r)\subset V^{r-1}~\forall r\in[1,k]\}.
$$
Let $\pi\colon \widetilde{F}\to N$, $\pi_V\colon \widetilde{F}_V\to E_V^n$ be natural projections. Set $G=GL(V)$, $\frakg=\mathfrak{gl}(V)=\Lie(G)$. Fix a primitive $e$th root of unity $\xi\in\bbC$. Consider the element $s\in G$ preserving each $V_i$ and acting by $\xi^i$ on $V_i$, $i\in I$. Note that the group $G\times \bbC^*$ acts on $\mathfrak{gl}(V)$ with $G$ acting by the adjoint action and $\bbC^*$ acting by the multiplication by scalars. Let $N\subset \frakg$ be the nilpotent cone. The $G\times \bbC^*$-action on $\frakg$ yields a $G\times \bbC^*$-action on $N$. Set $\widetilde s=(s,\xi^{-1})\in G\times \bbC^*$. We have
$\frakg^{\widetilde s}\simeq E_V$ and $N^{\widetilde s}\simeq E^n_V$.
Here the top index $\widetilde s$ always means the set of $\widetilde s$-stable points. 

Similarly, the $G$-action on $F$ yields a $G\times\bbC^*$-action on $F$, where $\bbC^*$ acts trivially. The $G\times \bbC^*$-action on $F$ and $N$ yields a $G\times \bbC^*$-action on $\widetilde F$. We have
$F^{\widetilde s}\simeq F_V$ and $\widetilde F^{\widetilde s}\simeq\widetilde F_V$.
The following lemma is an analogue of Conjecture \ref{conj_even}.
\begin{lem}
\label{lem_par-fib-QS}
For each $\mu\in \mathrm{VComp_e(\nu)}$ and each $x\in E_V^n$, we have $\mathrm{H}^{\rm odd}((\pi_\mu)^{-1}(x),\bbZ)=0$.
\end{lem}
\begin{proof}
The restriction of $\pi$ to the set of $\widetilde s$-stable points yields a morphism $\pi^{\widetilde s}\colon \widetilde F^{\widetilde s}\to N^{\widetilde s}$. We have the following commutative diagram
$$
\begin{CD}
\widetilde F_V @>\pi_V>> E^n_V\\
@|                 @|\\
\widetilde F^{\widetilde s}@>\pi^{\widetilde s}>> N^{\widetilde s}.
\end{CD}
$$
Then for each $x\in E^n_V$ we can identify the fibre $\pi_V^{-1}(x)$ with $(\pi^{\widetilde s})^{-1}(x)=(\pi^{-1}(x))^{\widetilde s}$.
By \cite[Thm.~3.9]{CLP} we have $\mathrm{H}^{\rm odd}((\pi^{-1}(x))^{\widetilde s},\bbZ)=0$. This implies the statement because for each $\mu\in \mathrm{VComp_e(\nu)}$ the variety $\widetilde{F}_V$ contains a connected component isomorphic to $\widetilde F_\mu$ and the restriction of $\pi_V$ to this component coincides with $\pi_\mu$.
\end{proof}

It is well-known that the set of $G_V$-orbits in $E_V^n$ is finite. This observation together with Lemma \ref{lem_par-fib-QS} motivates us to study the parity sheaves on $E_V^n$. As before, we fix an arbitrary field $\bfk$. Let $E_V^n=\coprod_{\lambda\in \Lambda_V^n}\calO_\lambda$ be the stratification by $G_V$-orbits.
\begin{lem}
There are no nontrivial $G_V$-equivariant local systems on $\calO_\lambda$ for each $\lambda\in \Lambda_V^n$.
\end{lem}
\begin{proof}
The proof is the same as the proof of Corollary \ref{coro_loc-sys-triv}.
\end{proof}
We can now prove that the stratified variety $E_V^n$ satisfies the condition (3.1) in the same way as in Lemma \ref{lem_prop*}.

Consider the following complexes in $D_{G_V}(E_V^n)$:
$$
^\delta\calL_\mu=(\pi_\mu)_*\bfk_{\widetilde F_\mu}[\dim_\bbC F_\mu],\quad ^\delta\calL_{V,\bfk}=\bigoplus_{\mu\in \mathrm{VComp_e(\nu)}}{^\delta\calL_\mu}.
$$
Let $\calQ_V$ be the full additive subcategory of all direct sums of shifts of direct summands of $^\delta\calL_{V,\bfk}$ in $D_{G_V}(E_V^n)$.

\begin{lem}
For each $\lambda\in\Lambda_V^n$, the parity sheaf $\calE(\lambda)$ exists. It is contained in $\calQ_V$.
\end{lem}
\begin{proof}
Consider the $G_V$-orbit $\calO_\lambda$. Let $\calO^\lambda$ be the $G$-orbit in $E_V$ that contains $\calO_\lambda$. We can find a flag type 
$\bfd=(d_1,\cdots,d_k)$ with $d_1,\cdots,d_k\in\bbN$ and $\sum_{r=1}^k d_k=\dim V$,
such that the restriction of $\pi$ to the connected component of $\widetilde F$ corresponding to $\bfd$ yields a resolution of singularities of $\overline{\calO^\lambda}$. After passing to $\widetilde s$-stable points in this resolution, we get a morphism of the form
$$
\coprod_{\mu\in C}\pi_\mu\colon\coprod_{\mu\in C}\widetilde F_\mu\to \overline{\calO^\lambda}^{\widetilde s}\subset E_V^{n}
$$
for some subset $C\subset \mathrm{VComp_e(\nu)}$.
Thus, there is a unique $\mu\in C\subset \mathrm{VComp_e(\nu)}$ such that $\pi_\mu$ induces an isomorphism $\pi_\mu^{-1}(\calO_\lambda)\to\calO_\lambda$.
Moreover, $\pi_\mu$ has even fibres, see Lemma \ref{lem_par-fib-QS}. Thus the complex $(\pi_\mu)_*\underline\bfk_{\widetilde F_\mu}$ is even, see Lemma \ref{lem_ev-resol}. To show that it contains a direct factor isomorphic to a shift of $\calE(\lambda)$ it is enough to prove that $\overline{\calO^\lambda}^{\widetilde s}$ can not contain a $G_V$-orbit $\calO_{\lambda'}$ such that $\calO_\lambda\subset \overline{\calO_{\lambda'}}$ and $\calO_\lambda\ne \calO_{\lambda'}$. Suppose that $\overline{\calO^\lambda}^{\widetilde s}$ contains such a $G_V$-orbit $\calO_{\lambda'}$. Let $\calO^{\lambda'}$ be the $G$-orbit containing $\calO_{\lambda'}$. Then $\calO_\lambda\subset \overline{\calO_{\lambda'}}$ implies $\calO^\lambda\subset \overline{\calO^{\lambda'}}$. Moreover, we have $\calO^\lambda\ne \calO^{\lambda'}$ because $\calO_\lambda$ and $\calO_{\lambda'}$ are not in the same $G$-orbit. (Otherwise, $\calO_\lambda$ and $\calO_{\lambda'}$ must have equal dimensions. This is not possible because $\calO_\lambda\subset \overline{\calO_{\lambda'}}\backslash \calO_{\lambda'}$.) Thus we get $\calO^\lambda\subset \overline{\calO^{\lambda'}}\backslash \calO^{\lambda'}$. Now we see that $\overline{\calO^\lambda}$ does not contain the $G$-orbit $\calO^{\lambda'}$. Thus $\overline{\calO^\lambda}^{\widetilde s}$ can not contain the $G_V$-orbit $\calO_{\lambda'}$.

Thus the parity sheaf $\calE(\lambda)$ exists and belongs to $\calQ_V$.
\end{proof}
\begin{coro}
The subcategories $\calQ_V$ and $\mathrm{Par}_{G_V}(E_V^n)$ of $D_{G_V}(E_V^n)$ coincide.
\end{coro}
\begin{df}
\label{def_QS}
For a given graded dimension vector $\nu\in\bbN I$ the \emph{quiver Schur algebra} is the graded algebra $A_{\nu,\bfk}=\Ext^*_{G_V}({^\delta\calL_{V,\bfk}},{^\delta\calL_{V,\bfk}})$,
where the $n$th graded component is given by the $n$th extension group.
\end{df}

\begin{rk}
It is clear from the definition that, for each $\mu\in \mathrm{VComp_e(\nu)}$, the algebra $A_{\nu,\bfk}$ contains an idempotent of degree zero $1_{\mu}$ such that $1_{\mu_1}A_{\nu,\bfk}1_{\mu_2}=\Ext_{G_V}^*(^\delta\calL_{\mu_2},^\delta\calL_{\mu_1})$.
\end{rk}

\begin{df}
For each $\mu\in \mathrm{VComp_e(\nu)}$, let $P_\mu$ be the graded projective $A_{\nu,\bfk}$-module $A_{\nu,\bfk}1_\mu$.
\end{df}

\begin{df}
Let $U_{e,\bbZ}^-$ be the integral form of the generic nilpotent Hall algebra of the quiver $\Gamma$, see \cite[Sec.~2.5]{SW}. For each $\nu\in\bbN I$, let $\bff_{\nu}\in U_{e,\bbZ}^-$ be the characteristic function of the semi-simple representation of $\Gamma$ with graded dimension vector $\nu$.
\end{df}

Let $\mathrm{proj}(A_{\nu,\bfk})$ be the category of graded projective finitely generated $A_{\nu,\bfk}$-modules. Let $K(A_{\nu,\bfk})$ be the split Grothendieck group of $\mathrm{proj}(A_{\nu,\bfk})$ viewed as an $\calA$-module.
Set $K(A_{\bfk})=\bigoplus_{\nu\in \bbN I}K(A_{\nu,\bfk})$ and $K(\calQ)=\bigoplus_{V}K(\calQ_V)$. For each $\nu_1,\nu_2\in \bbN I$, we have an algebra homomorphism $A_{\nu_1,\bfk}\otimes A_{\nu_2,\bfk}\to A_{\nu_1+\nu_2,\bfk}$, see \cite[Def.~2.8]{SW}. It yields an algebra structure on $K(A_\bfk)$. The following theorem is proved in \cite[Prop.~2.12]{SW}.
\begin{thm}
There is an $\calA$-algebra isomorphism
$K(A_\bfk)\to U_{e,\bbZ}^-$ such that $[P_\nu]\mapsto \bff_\nu$ for each $\nu\in\bbN I$,
where $\nu$ is viewed as a trivial vector multicomposition in $\mathrm{VComp_e(\nu)}$.
\end{thm}

Analogically to Theorem \ref{thm_cat-equiv-grad} we get the following result.
\begin{thm}
The functor
$Y\colon \calQ_V^{op}\to \mathrm{proj}(A_{\nu,\bfk})$ such that $\calF\mapsto \Ext_{G_V}^*(\calF,{^\delta\calL_{V,\bfk}})$
is an equivalence of categories.
\end{thm}

For $\nu_1,\nu_2\in\bbN I$, $\mu_1\in \mathrm{VComp_e(\nu_1)}$, $\mu_2\in \mathrm{VComp_e(\nu_2)}$ let $\mu_1\cup\mu_2\in  \mathrm{VComp_e(\nu_1+\nu_2)}$ denotes the concatenation of $\mu_1$ and $\mu_2$. We can also define an algebra structure on $K(\calQ)$ as in Section \ref{subs_new-bas} such that $^\delta\calL_{\mu_1}\circ{^\delta\calL_{\mu_2}}={^\delta\calL_{\mu_1\cup\mu_2}}$. We get the following result.
\begin{coro}
\label{coro_isom-Q-QSh}
There is an $\calA$-algebra isomorphism
$K(\calQ)\to U_{e,\bbZ}^-$ such that $^\delta\calL_\nu \mapsto\bff_{\nu}$ for each $\nu\in\bbN I$,
where $\nu$ is viewed as a trivial vector multicomposition in $\mathrm{VComp_e(\nu)}$.
\end{coro}
\begin{proof}
We need only to verify that the algebra structures on $K(\calQ)$ agrees with the algebra structure on $U_{e,\bbZ}^-$. This is true because $Y(^\delta\calL_\mu)=P_\mu$, $[P_{\mu_1}][P_{\mu_2}]=[P_{\mu_1\cup\mu_2}]$ (see \cite[Prop.~2.9]{SW}) and the elements $\bff_\nu$ for $\nu\in\bbN I$ generate $U_{e,\bbZ}^-$ as an $\calA$-algebra.
\end{proof}


In particular Corollary \ref{coro_isom-Q-QSh} yields an $\calA$-basis of $U^-_{e,\bbZ}$ in terms of parity sheaves.

\section*{Acknowledgements}
I would like to thank Eric Vasserot for his guidance and helpful
discussions during my work on this paper. I would also like to thank
Geordie Williamson for his comments on the paper and useful
discussions. I am also grateful to Catharina Stroppel for the suggestion to study the quiver Schur algebras in this paper.


\end{document}